\documentclass[sn-mathphys-num]{sn-jnl}

\usepackage{graphicx}%
\usepackage{multirow}%
\usepackage{amsmath,amssymb,amsfonts}%
\usepackage{amsthm}%
\usepackage{mathrsfs}%
\usepackage[title]{appendix}%
\usepackage{xcolor}%
\usepackage{textcomp}%
\usepackage{manyfoot}%
\usepackage{booktabs}%
\usepackage{algorithm}%
\usepackage{algorithmicx}%
\usepackage{algpseudocode}%
\usepackage{listings}%
\usepackage{latexsym}
\usepackage{epsfig}
\usepackage{calrsfs}

\DeclareMathAlphabet{\pazocal}{OMS}{zplm}{m}{n}

\newcommand{\maptoR}[1]{\rightarrow\mathbb{R}^{#1}}

\theoremstyle{thmstyleone}%
\newtheorem{theorem}{Theorem}[section] 
\newtheorem{prop}[theorem]{Proposition}%
\newtheorem{lem}[theorem]{Lemma}

\theoremstyle{thmstyletwo}%
\newtheorem{rem}{Remark}[section]%
\newtheorem{conv}{Convention}[section]
\newtheorem{notn}{Notation}[section]

\theoremstyle{thmstylethree}%
\newtheorem{dfn}{Definition}%

\begin{document}

\title[Article Title]{Length-minimizing LED trees}


\author*[1]{\sur{Mariana Sarkociov\'{a} Reme\v{s}\'{\i}kov\'{a}}}\email{mariana.sarkociova@gmail.com}

\author[1]{\sur{Peter Sarkoci}}\email{sarkoci@math.sk}

\author[2]{\sur{M\' aria Trnovsk\' a}}\email{Maria.Trnovska@fmph.uniba.sk}

\affil[1]{\orgdiv{Department of Mathematics and Constructive Geometry}, \orgname{Faculty of Civil Engineering, Slovak University of Technology in Bratislava}, \orgaddress{\street{Radlinsk\' eho 11}, \city{Bratislava}, \postcode{81005}, \country{Slovakia}}}
\affil[2]{\orgdiv{Department of Applied Mathematics and Statistics}, \orgname{Faculty of Mathematics, Physics and Informatics}, \orgaddress{\street{Mlynsk\' a Dolina F1}, \city{Bratislava}, \postcode{84248}, \country{Slovakia}}}

\abstract{In this paper, we introduce a specific type of Euclidean tree called LED (Leaves of Equal Depth) tree. LED trees can be used in computational phylogeny, since they are a natural representative of the time evolution of a set of species in a feature space. This work is focused on LED trees that are length minimizers for a given set of leaves (species) and a given isomorphism type (the hierarchical structure of ancestors). The underlying minimization problem can be seen as a variant of the classical Euclidean Steiner tree problem. Even though it has a convex objective function, it is rather non-trivial, since it has a non-convex feasible set. The main contribution of this paper is that we provide a uniqueness result for this problem. Moreover, we explore some geometrical and topological properties of the feasible set and we prove several geometrical characteristics of the length minimizers that are analogical to the properties of Steiner trees. At the end, we show a simple example of an application in historical linguistics.}



\maketitle

\section{Introducing LED trees}\label{secIntroduction}

A LED tree is a special type of Euclidean graph that is built from a {\it rooted tree} -- a graph $G=(V,E,R)$, where $R\in V$ is designated as the root. A {\it Euclidean rooted tree} can be defined as a pair $\Psi=(G,\psi)$, where $G$ is a rooted tree and $\psi\colon V\maptoR{n}$ is an arbitrary map. The vertices of $\Psi$ are the images of the elements of $V$ and the role of the root is naturally assigned to the vertex $\psi(R)$. Since $\mathbb{R}^n$ is equipped with its standard metric, we can measure the length of any edge of $\Psi$ by evaluating the distance between its adjacent vertices. 

The definition of a LED tree relies on several graph-theoretical concepts. Since this paper does not have a purely graph-theoretical character and it might have audience from different fields, we will now briefly mention these concepts together with some other ones that will be widely used throughout the text. 

\begin{itemize}
	\item{A {\it leaf} of a tree is any vertex of degree one (with just one adjacent edge). Any other vertex is called an {\it inner vertex}.}
	\item{The {\it root path} of a vertex of a rooted tree is the unique path that connects the vertex with the root.}
	\item{The {\it depth} of a vertex of a rooted tree is the length of its root path. In the non-Euclidean case, this is just the number of edges included in the path and for a Euclidean rooted tree, the depth is the standard Euclidean length of the path. } 
	\item{A {\it leaf path} is any downward path that connects the vertex with a leaf.  A {\it downward path} is any path where each vertex has a higher depth than its predecessor. }
	\item{The {\it height} of a {\it vertex} is the length of its longest leaf path. The {\it height} of the {\it tree} is the height of its root.}
	\item{The {\it parent} of a vertex is its successor on the root path. A {\it child} of a vertex is its successor on any of its leaf paths. }
	\item{The set of all Euclidean trees can be split into equivalence classes, where two Euclidean trees  $\Psi_1=(G_1,\psi_1)$ and $\Psi_2=(G_2,\psi_2)$ are equivalent, if $G_1$ and $G_2$ are isomorphic. A corresponding equivalence class is called {\it isomorphism type} of the tree.}
\end{itemize}

Now we are ready to introduce the key concept of our work.

\begin{dfn}\label{DefLEDTree}
	A Euclidean rooted tree is called a {\em LED (Leaves of Equal Depth) tree}, if all its leaves have the same depth.
\end{dfn}

Some examples of two-dimensional LED trees are shown in Fig. \ref{FigLEDTrees}. In further sections of this text, we will explore in detail some of their interesting properties. Within the introduction, we would like to point out just one of them, which we consider to be the most striking to the eye. Let us take any vertex $v$ and let us consider the subtree containing all leaf paths of $v$. If we consider $v$ to be the root of that subtree, then the subtree is actually also a LED tree. Thus any LED tree has a recursive structure -- it contains as many LED subtrees as it has vertices. 

\begin{figure}[h]
	\centering
	\includegraphics[width=0.3\textwidth]{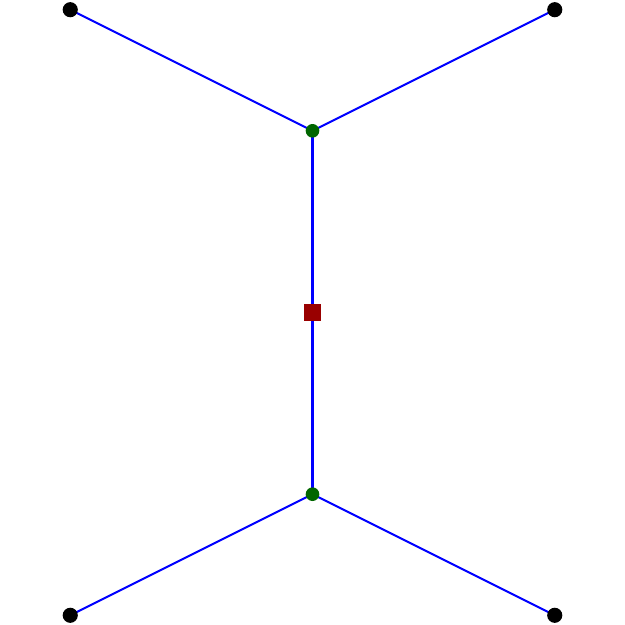}
	\includegraphics[width=0.3\textwidth]{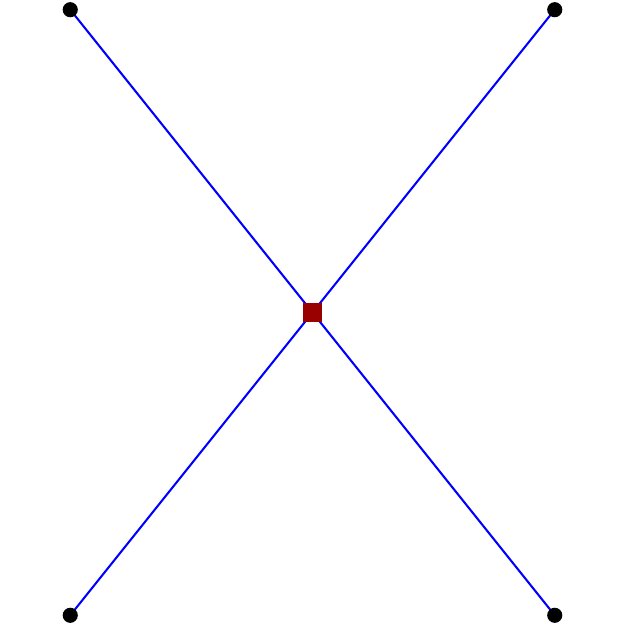}
	\includegraphics[width=0.3\textwidth]{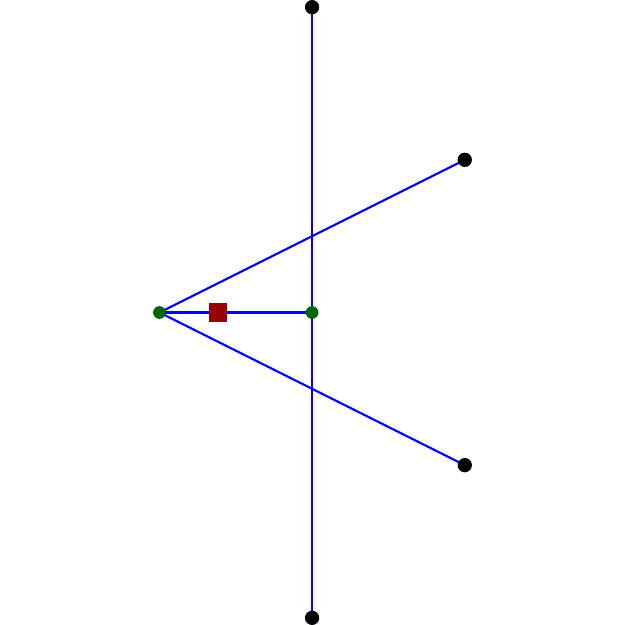} \vspace{0.3cm} \\
	\includegraphics[width=0.3\textwidth]{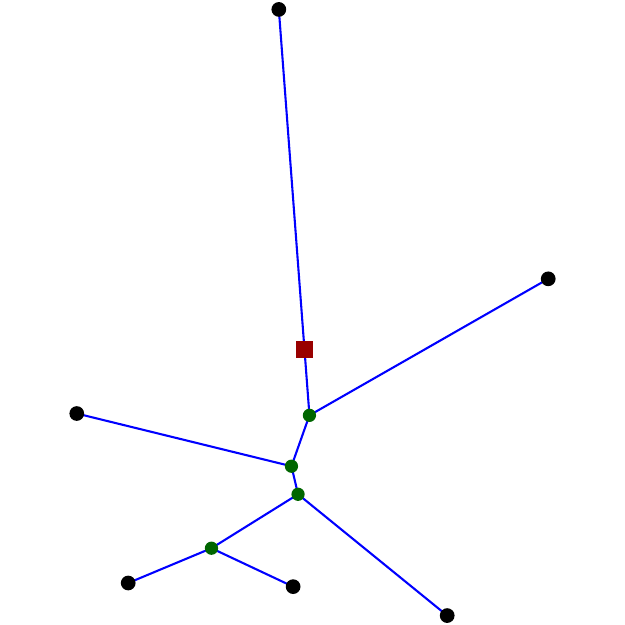}
	\includegraphics[width=0.3\textwidth]{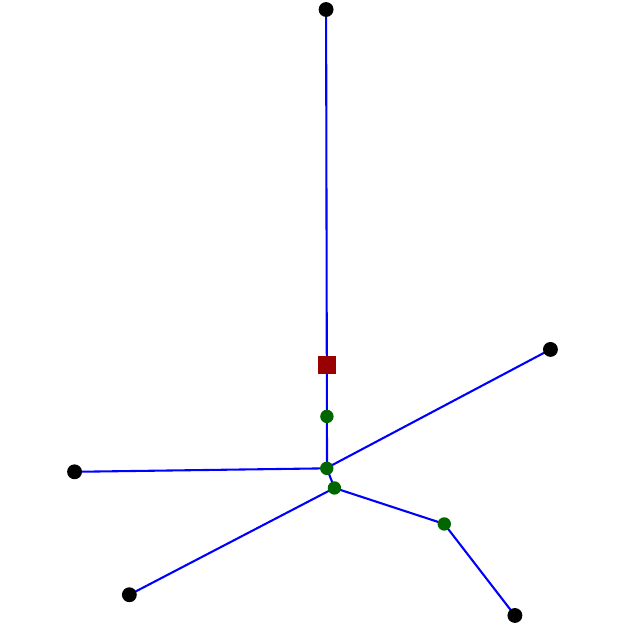}
	\includegraphics[width=0.3\textwidth]{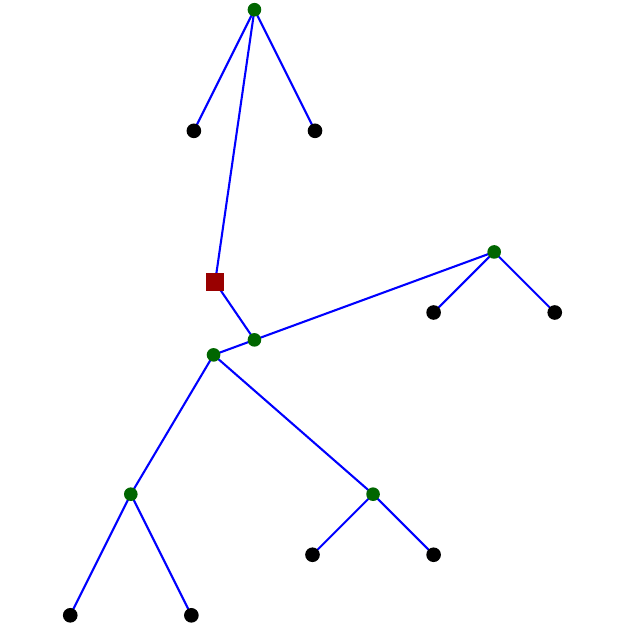}
	\caption{Examples of two-dimensional LED trees. The root of each tree is marked by a square.}
	\label{FigLEDTrees}
\end{figure}

When researching relevant sources, we did not find any that deal with exactly the same type of trees as we present. The most closely related concept that we came across is the S-graph introduced by Buckley and Lewinter \cite{BL}. The definition of this type of graph is based on the notion of eccentricity. In any kind of graph, the {\it eccentricity} of a vertex $v$ is defined as the maximum graph distance between $v$ and any other vertex in the graph. In general, there can be an arbitrary number of vertices located at the maximum distance from $v$; all of them are called vertices {\it eccentric} to $v$. Further, a vertex of minimum eccentricity is called {\it central} and a vertex of maximum eccentricity is called {\it peripheral}. In an {\it S-graph}, the set of peripheral vertices is the same as the set of vertices which are eccentric to any of the central vertices. Now if we consider only trees, then all vertices eccentric to any given vertex are always leaves. In a LED tree, all leaves have the same eccentricity and they are all peripheral. At the same time, there is no other peripheral vertex. A central vertex is the image of any element of $V$ that maps onto $\psi(R)$. The set of vertices eccentric to central vertices is again the set of leaves. Therefore, a LED tree is always a Euclidean S-graph (or, more specifically, S-tree). However, the two concepts are not equivalent, since not all Euclidean S-trees are LED trees.  Moreover, even though we found a few other authors who studied S-graphs \cite{AlAd,Gli1,Gli2,Kys}, so far we have not found any research exploring them in the Euclidean setting or specifically in the realm of trees. 

To conclude the introduction, we define another important concept that will help us explain the topic this work. 

\begin{dfn}\label{defHangingType}
	Let $\Psi=(G_{\Psi},\psi)$ and $\Phi=(G_{\Phi},\phi)$ be two Euclidean trees (rooted or not) and let $L_{\Psi}$ and $L_{\Phi}$ be the leaf sets of $G_{\Psi}$ and $G_{\Phi}$. Further, let $\psi_L=\psi\restriction_{L_{\Psi}}$ and $\phi_L=\phi\restriction_{L_{\Phi}}$. We define the following equivalence relation on the set of all Euclidean trees: $\Psi\sim\Phi$ if and only if there is a graph isomorphism $f$ mapping the vertices of $G_{\Psi}$ to the vertices of $G_{\Phi}$ such that 
	\[ \psi_L=\phi_L\circ\left(f\restriction_{L_{\Psi}}\right). \]
	A corresponding equivalence class is called a {\em hanging type}. 
\end{dfn}

The term ``hanging type'' arose from a visual image of an element of an equivalence class given by Definition \ref{defHangingType}: we take a tree of a given isomorphism type and we ``hang it by leaves'' in a Euclidean space -- its leaves are fixed to prescribed positions. The rest of the vertices are not fixed and they can be placed arbitrarily. 



\section{LED trees as chronograms}\label{secLEDTreesAsChronograms}

This work is focused mostly on a special type of LED tree that has a length minimizing property. This type of tree emerged in a practical application that we started to deal with and it soon caught our interest also from the theoretical point of view. The main contribution of the work presented here are the theoretical results that we obtained, however, the story would not be complete without its practical background and an example where our results were actually applied. The former is discussed in this section and the latter will conclude our presentation. 

LED trees naturally emerge in situations when we try to model the evolution of a set of species assumed to have a common ancestor (or a hierarchically ordered set of ancestors). In such cases, the hierarchical structure of the species and their ancestors can be represented by a rooted phylogenetic tree. In some situations, the pure hierarchy might not be enough and we might need to represent the real temporal evolution. This can be done by constructing a tree in which the lengths of the branches represent the duration of the corresponding part of the evolution. This type of phylogenetic tree is called a {\it chronogram}. Now let us assume that all leaves of the chronogram represent species existing at the same time (i.e. no species come to extinction before the end of the evolution) and that the speed of evolution of any species is constant through the whole time interval. If the chronogram is placed in a Euclidean space, then it is a LED tree. 

Phylogenetic trees are used mostly in biology and related fields, but they also appear in other scientific or technical disciplines. Our motivation to explore LED trees actually comes from linguistics. Phylogenetic trees, both basic and chronograms, are often used by linguists to represent the evolution of languages from their (often hypothetical) ancestral languages to the contemporary state. Basic phylogenetic trees are used for hierarchical classification of languages into families -- i.e. the Slovak language is the member of the Slavic family, then the Balto-Slavic family and finally the Indo-European family. A corresponding chronogram would contain also the information about the time that passed since the Proto-Indo-European or Proto-Slavic era. Since the exact dating of some milestones of the evolution of various language families is still not completely clear, many researchers design algorithms for constructing chronograms in order to estimate the unknown time periods. Even though this approach does sometimes get criticized by linguists as too abstract and simplifying, it remains in the field of interest and some widely accepted results are based on it (among others, Bouckaert et al. \cite{BA}, Chang et al. \cite{Chang}, Gray and Atkinson \cite{GA}, Kassian et al. \cite{Kass}).

A general LED chronogram described above is simply a representation of a time evolution that correctly displays the length of any time interval within the process. The positions of leaves do not necessarily carry any important information and the leaves can be placed arbitrarily according to the needs of the visualization. In our work, we consider a more specific setting and we assume that the coordinates of the leaves in the Euclidean space represent some actual features of the species. This means that their Euclidean distance measures how far they diverged from each other during the evolution. This can help us estimate the hanging type of the tree -- we can assume that species situated close to each other differentiated later than more distant species. Also the chronogram itself can be seen as a representative of the actual evolution through the feature space. Having given the species as the final stage of the evolution, we can then explore various scenarios how they could get to their positions.

Let us consider the situation in Fig. \ref{FigEvolutionScenarios}, where we have five species $A$, $B$, $C$, $D$, $E$ and several different evolution scenarios. The first picture represents the setting when there is just one common ancestor for all species and they differentiate immediately at the beginning of the evolution. This is the minimal time scenario, i.e. the shortest evolution that we can have. The second picture represents another situation. Based on the distances of the species, we assumed that the species $B$ and $C$ had a common ancestor $Y$, which in turn had a common ancestor with $A$ denoted by $Z$. The species $D$ and $E$ had a common ancestor $X$. Finally, $X$ and $Z$ are direct descendants of $R$, which is the root ancestor representing the beginning of the evolution. The tree depicted in the picture has no particular properties, it is just a generic tree of the given hanging type. 

Another tree of the same hanging type is the one shown in the third picture. But this time it has a specific feature -- among all trees of the same hanging type, this is the one with the minimum length. Let us think about the type of evolution that it represents. The length $\Lambda$ of the tree can be expressed as
\begin{eqnarray*}
	\Lambda & = & |AZ|+|BY|+|CY|+|YZ|+|ZR|+|DX|+|EX|+|XR| \\
	& = & (|AZ|+|ZR|)+(|DX|+|XR|)+(|BY|+|YZ|)+|CY|+|EX|.
\end{eqnarray*}
Since we are dealing with a LED tree, the sums in the first two pairs of parentheses are equal. Thus we have
\[ \Lambda=2(|AZ|+|ZR|)+(|BY|+|YZ|)+|CY|+|EX|. \]
Here, the term $|AZ|+|ZR|$ represents the overall time of the evolution. The sum $|BY|+|YZ|$ represents the time needed for $Z$ to evolve into $B$ or $C$. Further, $|CY|$ measures the time in which $Y$ evolves into $B$ or $C$ and similarly $|EX|$ represents the time taken by the evolution of $X$ into $D$ or $E$. So, for a given hanging type, minimizing $\Lambda$ can be seen as minimizing the overall time of the evolution with the condition that the evolution after each split is also as short as possible. 

In the fourth picture, we can see a tree with still the same hanging type but another particular property. This time all ancestor species lie on the line segment connecting their two direct descendants. In terms of the species evolution, this can be interpreted as follows: all features that have the same value in the two descendants, must have the exact same value in the ancestor. For the evolutions described above, this need not be the case -- the values of such features can still change after the differentiation, but they will eventually evolve to be equal. 

The fifth picture is the length minimizer for another hanging type, where we assumed that $A$ and $B$ differentiate only after $C$ differentiates from their common ancestor. After the split, $A$ and $B$ diverge quickly, while $B$ stays closer to $C$. This is somewhat less probable, but not an impossible scenario. 

Yet another hanging type is shown in the last picture. Here $C$ is differentiated from the very beginning, while the other species differentiate later. Also, we can see that $A$ has a common ancestor with $D$, although it is actually closer to $B$ and $C$ and we have a similar situation with $B$ and $E$. Moreover, after $A$ and $D$ differentiate, they, in some sense, go backwards to the root -- their distance from the root decreases for some time. Again, this situation cannot be ruled out completely, but we consider it much less likely than any other scenario mentioned before. 

\begin{figure}[h]
	\centering
	\includegraphics[width=0.32\textwidth]{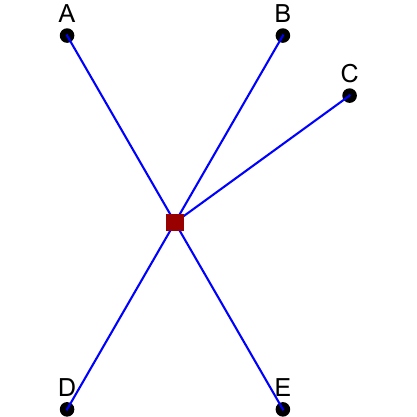}
	\includegraphics[width=0.32\textwidth]{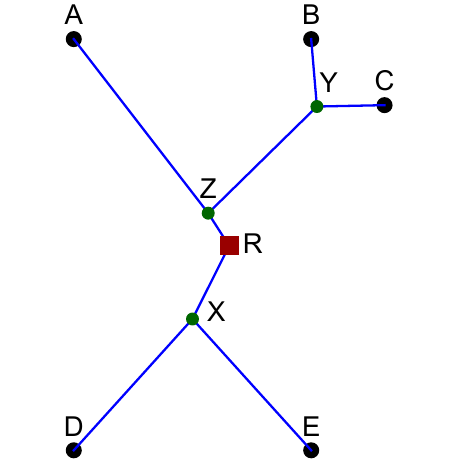}
	\includegraphics[width=0.32\textwidth]{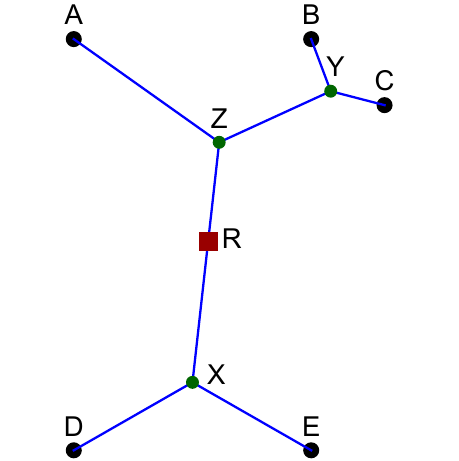} \vspace{0.3cm} \\
	\includegraphics[width=0.32\textwidth]{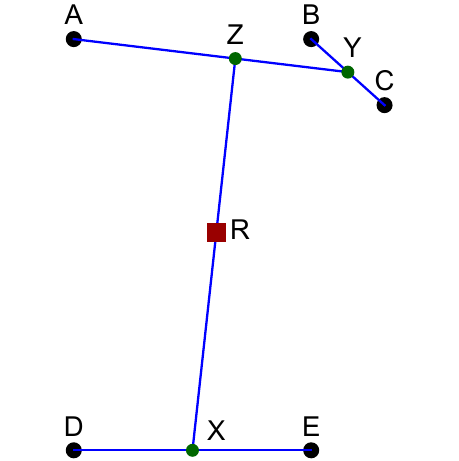}
	\includegraphics[width=0.32\textwidth]{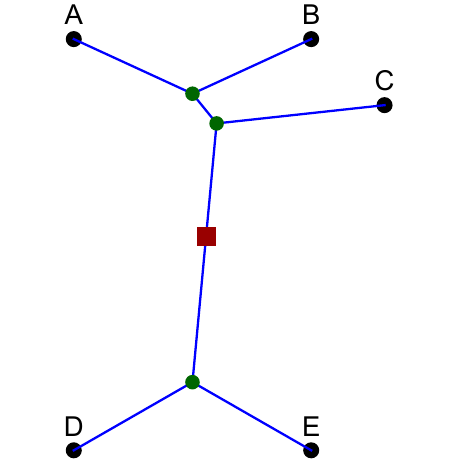}
	\includegraphics[width=0.32\textwidth]{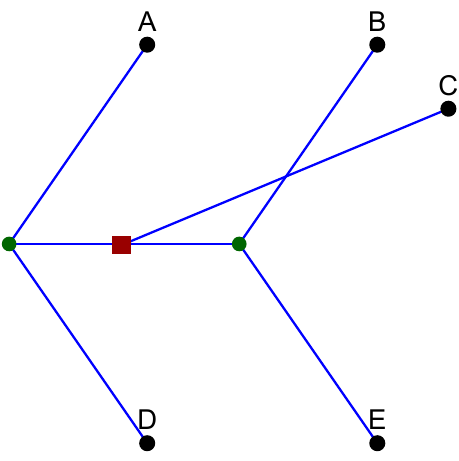}
	\caption{Six different possible scenarios for evolution of the species $A$, $B$, $C$, $D$ and $E$ from their root ancestor (marked by a square).}
	\label{FigEvolutionScenarios}
\end{figure}


\section{Overview of the results}\label{secOverview}

As we indicated in the previous section, when trying to construct a chronogram for a given set of species (leaves), we can often determine the hanging type (or several most probable hanging types) from the positions of the species in the Euclidean space. Therefore, our task is reduced to finding a reasonable chronogram of a given hanging type. Among these chronograms, there is usually no one and only correct choice, but more often a set of probable/acceptable scenarios. Therefore, rather than picking the best tree, we can offer several possibilities along with an explanation of what they represent. As we have shown, one type of LED tree that carries a clearly defined information is a length-minimizing LED tree, an example of which is provided in the third and the fifth picture of Fig. \ref{FigEvolutionScenarios}. This type of tree will be studied in the rest of the paper.

Finding length-minimizing LED trees is closely related to the Euclidean Steiner tree problem, where we search for a minimum length tree for a given set of leaves, with no further requirements (for an explanation and overview of some results, see e.g. Brazil \cite{Brazil1}). This problem is most often solved in plane, but there are also results for more dimensions (to name just a few: Fampa \cite{Fampa}, Ouzia and Maculan \cite{Ouzia}, Smith \cite{Smith}, Snyder \cite{Snyder}). Just as LED trees, Steiner trees have been used for finding plausible phylogenetic trees, though not chronograms \cite{Brazil,Weng}. It is not difficult to find examples of situations, when the Euclidean Steiner tree and the length-minimizing LED tree are the same, except that the Steiner tree is not rooted. Some of them are shown in Fig. \ref{FigSteiner}. However, in most cases these trees differ. 

Of course, when looking for a Steiner tree, the hanging type is usually not known and that is what makes this problem difficult; with a given hanging type it would be just an unconstrained convex optimization. In our case, the LED property brings in some additional difficulty and even when the hanging type is prescribed, the problem needs some analysis before attempting to solve it by an optimization algorithm. In particular, as we will see later, the corresponding optimization problem has a rather complex non-convex feasible set and this naturally leads to expecting local extrema and saddle points. However, it comes as a surprise that under a certain (and not very limiting) regularity assumption, a stationary point of the objective function is unique. This is what we consider to be the most interesting result of our work. The corresponding propositions are formulated and proved in Sec. \ref{secOptimalityOfStationaryPoints}. Besides that, we explore some topological and geometrical properties of the feasible set in Sec. \ref{secTheFeasibleSet} and we prove several geometrical properties of length-minimizing LED trees that can be observed also in Steiner trees (Sec. \ref{secStationaryPoints} and \ref{secOptimalityOfStationaryPoints}). Finally, we present some simple examples of using the length-minimizing LED trees for reconstruction of the evolution of Indo-European languages. 

\begin{figure}[h]
	\centering
	\includegraphics[width=0.24\textwidth]{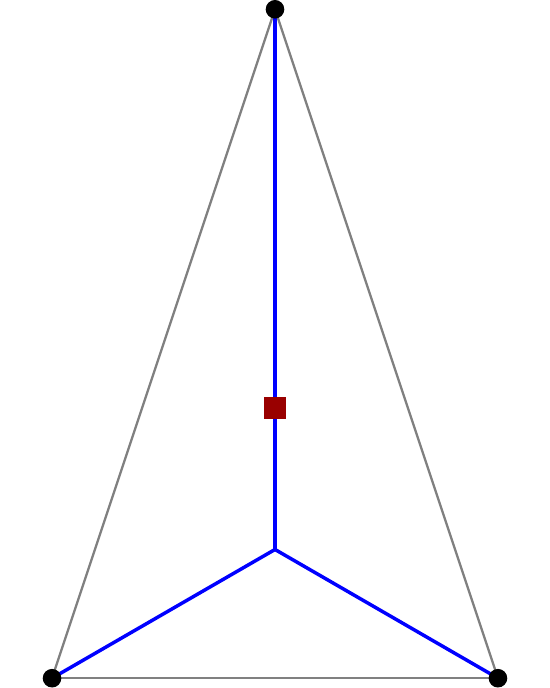}
	\includegraphics[width=0.24\textwidth]{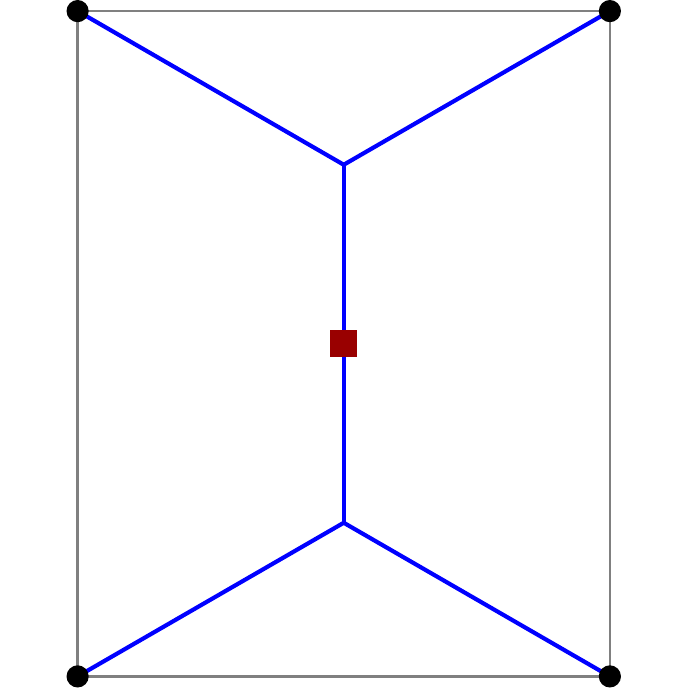}
	\includegraphics[width=0.24\textwidth]{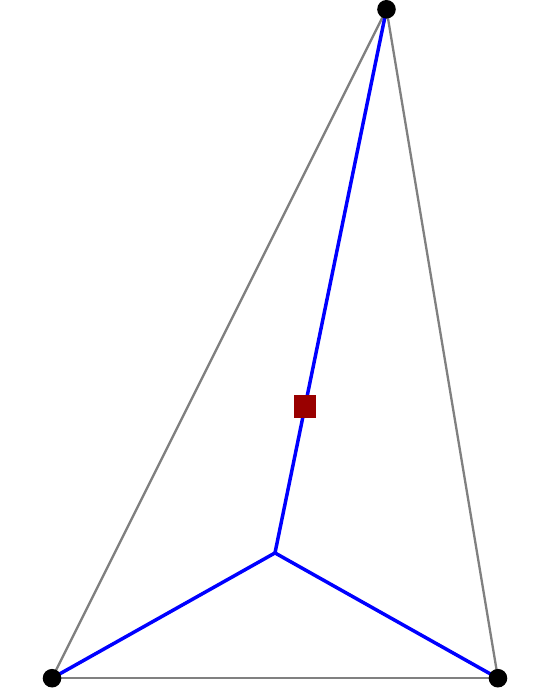} 
	\includegraphics[width=0.24\textwidth]{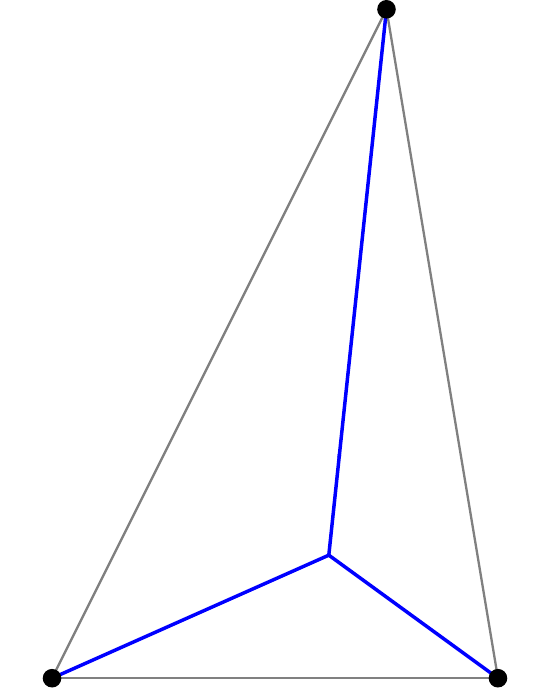} 
	\caption{Lenght-minimizing LED trees and Euclidean Steiner trees. In the first two situations (an isosceles triangle and a rectangle), both trees are the same, except for the fact that the LED tree is rooted. In the third case, the three leaves are vertices of a general triangle and the corresponding Steiner tree  is different from the length-minimizing LED tree.}
	\label{FigSteiner}
\end{figure}


\section{Full binary LED trees and their basic properties}\label{secFullBinaryLEDTrees}

Before we proceed to a deeper exploration of LED trees and formulation of our results, let us first make one restricting assumption. The definition \ref{DefLEDTree} does not impose any special conditions on the tree $G$, from which the LED tree is constructed. However, in the rest of the paper, we will limit ourselves only to LED trees that are constructed from a {\it full binary tree} -- a tree where each vertex has either two or zero children. Such trees will be called {\it full binary LED trees}. This setting is more suitable for chronograms and it also simplifies our reasoning and formulations. Trees with more than two branches at some of the forks are still kept in the picture, since the map $\psi$ does not have to be injective. That means that several vertices from $V$ can be mapped onto the same point in $\mathbb{R}^n$ and some edges can have their length equal to zero. Thus, even though the underlying structure is binary, the corresponding Euclidean tree can represent a non-binary evolution. What will be really left out are the trees with single child vertices. But this is actually what we aim for; in a phylogenetic tree, such vertices would be completely redundant.

For a full binary LED tree, we are able to exactly specify some important quantities. Let us suppose that we have $n_l$ points in $\mathbb{R}^n$ and we want to construct a full binary LED tree with these points as leaves. Then, no matter the hanging type, the tree will always have
\begin{itemize}
	\item{$n_v$ inner vertices, where $n_v=n_l-1$, }
	\item{$n_t$ total vertices, where $n_t=2n_l-1$,}
	\item{$n_e$ edges, where $n_e=2(n_l-1)$.}
\end{itemize}

Now let us have a closer look at the structure of a full binary LED tree. For our further purposes, it will be useful to realize how exactly this tree is built up of LED trees of lower height. To make the explanation more clear, the idea is also illustrated in Fig. \ref{FigBuildingLEDTree}. Let us consider a given full binary LED tree and its arbitrary LED subtree that has at least three vertices. Let $v$ be the root of the subtree and $v_1$, $v_2$ its two children; let us assume that $v_1\neq v_2$. Let $h(v_1)$ and $h(v_2)$ be the heights of $v_1$ and $v_2$. Then, since $v$ is the root of a LED tree, we must have
\[ \|v-v_1\|+h(v_1)=\|v-v_2\|+h(v_2), \]
which means
\begin{equation}
	\|v-v_1\|-\|v-v_2\|=h(v_2)-h(v_1).
	\label{eqHyperboloid}
\end{equation}
This implies that $v$ lies on the two-sheet rotational hyperboloid with foci $v_1$, $v_2$ and with the semi-major axis equal to $\left|h(v_2)-h(v_1)\right|/2$. More specifically, it lies on the sheet of the hyperboloid determined by the sign of $h(v_2)-h(v_1)$: if this value is positive, then it will be the sheet whose vertex is closer to $v_2$. If it is negative, then it will be the other sheet. If it is zero, then the hyperboloid is actually just a single hyperplane.

In the special case when $v_1=v_2$, the position of $v$ is not restricted by any condition. The vertex $v$ can lie anywhere and once we connect it with $v_1$ and $v_2$, we will always get a LED tree. Of course, the edges $vv_1$ and $vv_2$ will coincide.

Let us also remark that in any case it holds
\begin{equation}
	\|v_1-v_2\|\geq |h(v_2)-h(v_1)|.
	\label{eqHypExistence}
\end{equation}
If this was not true, then there would be no way to find a position for $v$ that would satisfy (\ref{eqHyperboloid}). 
\begin{figure}[h]
	\centering
	\includegraphics[width=0.75\textwidth]{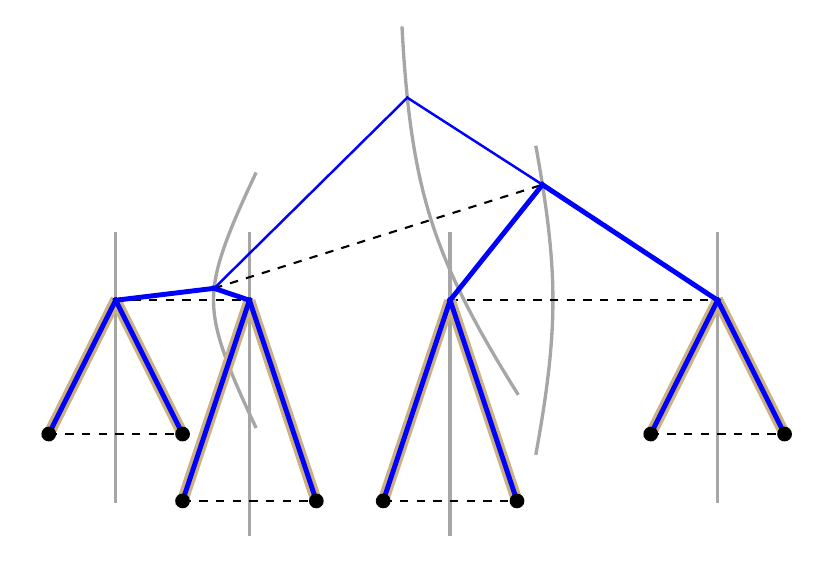}
	\caption{The structure of a LED tree. We can see that this particular tree consists of 15 LED subtrees. One of them is the tree itself, eight of them are the leaves and the remaining six are highlighted by thicker lines. We can also see that each inner vertex lies on the hyperbola whose foci are the children of that vertex and the semi-major axis is given by the difference of heights of the children. For vertices, whose children are leaves, the hyperbola is a straight line.}
	\label{FigBuildingLEDTree}
\end{figure}

Later on, we will use the above mentioned facts when trying to imagine and parametrize the set of all LED trees of a given hanging type. But before we get into anything else, let us introduce some notations that will be used throughout the text. 

\begin{notn}\label{notBasic}
In what follows,
	\begin{itemize}
		\item{$v_i$, $i=1,\dots ,n_t$, denotes the $i$-th vertex of the tree, while the values $i=1,\dots,n_v$ correspond to the inner vertices and the values $i=n_v+1,\dots,n_t$ represent the leaves,}
		\item{$e_j$, $j=1,\dots ,n_e$, stands for the $j$-th edge,}
		\item{$r$ is the index of the root.}
	\end{itemize}
For whichever values of $i$ it makes sense, 
	\begin{itemize}
		\item{$v_i^L$ is the child of $v_i$ with the lower index, $v_i^R$ is the child of $v_i$ with the higher index and $v_i^U$ denotes the parent of $v_i$,}
		\item{$u_i^L$, $u_i^R$ and $u_i^U$ are the unit vectors in directions of $v_i-v_i^L$, $v_i-v_i^R$ and $v_i-v_i^U$, whenever these differences are non-zero,}
		\item{$i_L$, $i_R$ and $i_U$ are the indices of the edges connecting $v_i$ with $v_i^L$, $v_i^R$ and $v_i^U$, respectively.}
	\end{itemize}
\end{notn}


\section{Length-minimizing LED trees}\label{secLengthMinimizingLEDTrees}

At this point, we can start discussing the main object of interest of our work. Let $\pazocal{L}(\pazocal{H})$ represent the set of all LED trees of hanging type $\pazocal{H}$. We are looking for a tree that solves the minimization problem
\begin{equation}
	\min_{\Psi\in\pazocal{L}(\pazocal{H})} \Lambda(\Psi),
	\label{eqMinProblemBasic}
\end{equation}
where $\Lambda\colon\pazocal{L}(\pazocal{H})\rightarrow\mathbb{R}$ and $\Lambda(\Psi)$ is the length of the tree $\Psi$, i.e.
\[ \Lambda(\Psi)=\sum\limits_{j=1}^{n_e} \|e_j\|. \]

In order to be able to apply some standard optimization technique, it is convenient to rewrite (\ref{eqMinProblemBasic}) as a constrained minimization problem in a Euclidean space. This is possible, since any LED tree can be represented by a point in $\mathbb{R}^{n_v n}$, whose $i$-th $n$-tuple contains the coordinates of $v_i$. In the rest of the text, this representation will be often mentioned alongside the original LED tree. Therefore, it will be useful to keep in mind the following notation.
\begin{notn}\label{notBeta}
	If $\beta$ is a point in $\mathbb{R}^{n_v n}$ representing a LED tree, then this tree will be refered to as $\Psi_{\beta}$. 
\end{notn}

Now let us define functions $\lambda_j\colon\mathbb{R}^{n_v n}\rightarrow\mathbb{R}$, $j=1,\dots, n_e$, where $\lambda_j(\beta)$ is the length of the $j$-th edge of $\Psi_{\beta}$. Further, let $\pazocal{J}_i$, $i=n_v+1,\dots, n_t$, be the index set consisting of the indices of all edges that are contained in the root path of the leaf $v_{i}$. Then we are looking for $\beta\in\mathbb{R}^{n_v n}$ solving the optimization problem
\begin{equation}
	\begin{array}{ll}
		\displaystyle \min_{\beta\in\mathbb{R}^{n_v n}} & \lambda(\beta) \vspace{0.1cm}\\
		\mathrm{subject\,\, to}& \displaystyle \sum\limits_{\iota\in\pazocal{J}_{i}} \lambda_{\iota}(\beta)-\sum\limits_{\iota\in\pazocal{J}_{i_0}} \lambda_{\iota}(\beta)=0, \quad i=n_v+1,\dots,n_t,\,i\neq i_0, 
	\end{array}
	\label{eqConstraintsLeaves}
\end{equation}
where $i_0$ is any fixed leaf index and $\lambda\colon\mathbb{R}^{n_v n}\rightarrow\mathbb{R}^n$ is defined as
\[ \lambda(\beta) =\sum\limits_{j=1}^{n_e} \lambda_j(\beta). \]
The $n_l-1$ constraints (\ref{eqConstraintsLeaves}) represent the property of equal depth of all leaves.

The constraints can be written in a more economic way, if we focus on inner vertices instead of leaves. Note that any two leaf-root paths meet at some inner vertex and from then on they pass through the same vertices. Therefore, some of the edge lengths in the equalities (\ref{eqConstraintsLeaves}) appear in one equality with both positive and negative sign and thus we have some redundancy in our expressions. A more to-the-point way of saying the same thing would be to assemble equalities containing edge lengths just up to the meeting point of two leaf-root paths. At any inner vertex, we can have a lot of leaf paths, but it is always possible to select two and only two without any common edges. If we do this for every inner node and write the corresponding equalities, we will get $n_v=n_l-1$ conditions equivalent to the constraints (\ref{eqConstraintsLeaves}).

Now, the question is, how to choose the two aforementioned leaf paths for a given inner vertex. The answer is that they can be chosen arbitrarily, but for the sake of simplicity and clarity, we will choose them in the same defined way for each of the inner vertices. Let us have an arbitrary inner vertex $v_i$. The first path will go ``to the left'' -- the first edge on it will be $e_{i_L}$ and in any inner vertex $v_k$ that will follow, we will always proceed through the edge $e_{k_L}$. We will call this path the {\it left leaf path} of the vertex $v_i$. The second path will start ``to the right'' -- with the edge $e_{i_R}$ -- but from then on, we will be turning left and every inner vertex $v_k$ on the path will be followed by the edge $e_{k_L}$. This path will be called the {\it right leaf path} of $v_i$. For a complete clarity, the idea is illustrated in Fig. \ref{FigLeafPaths}.

\begin{figure}[h]
	\centering
	\includegraphics[width=0.32\textwidth]{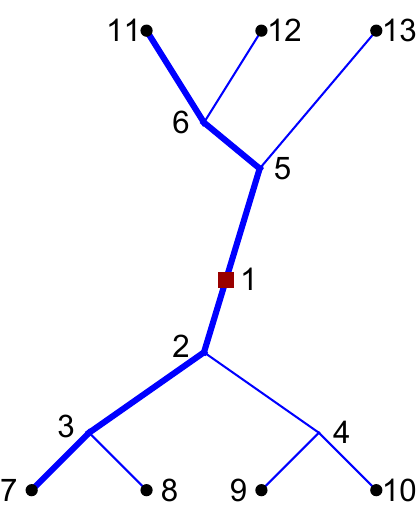}
	\includegraphics[width=0.32\textwidth]{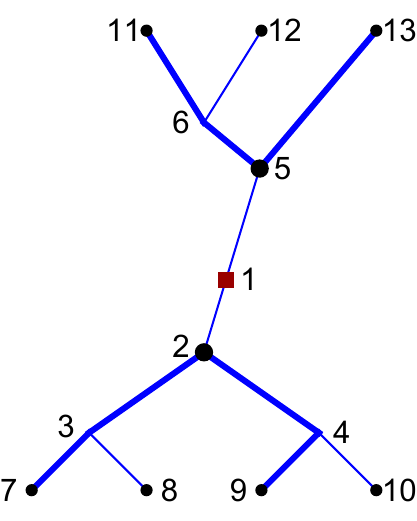}
	\includegraphics[width=0.32\textwidth]{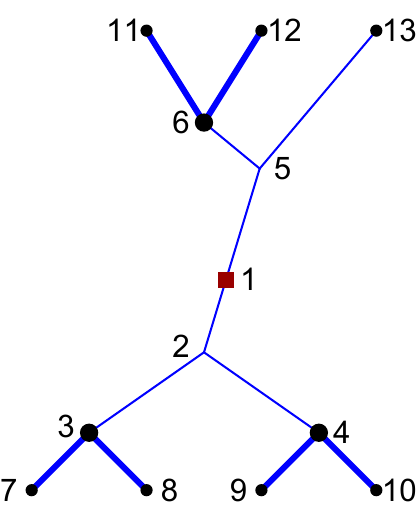} 
	\caption{Left and right leaf paths for all inner vertices. These paths are used to constitute the constraints of the minimization problem (\ref{eqOriginalMinProblem}).}
	\label{FigLeafPaths}
\end{figure}

To assemble the corresponding formulation of the minimization problem, let $v_i$ be an arbitrary inner vertex. Let $\pazocal{I}_i^L$ and $\pazocal{I}_i^R$ be the sets of indices of all edges contained in the left and right leaf path of $v_i$, respectively. Then our problem reads
\begin{equation}
	\begin{array}{ll}
		\displaystyle \min_{\beta\in\mathbb{R}^{n_v n}} &\lambda(\beta), \vspace{0.1cm}\\
		\mathrm{subject\,\, to} & \displaystyle\sum\limits_{\iota\in\pazocal{I}_i^L} \lambda_{\iota}(\beta)-\sum\limits_{\iota\in\pazocal{I}_i^R} \lambda_{\iota}(\beta)=0, \quad i=1,\dots,n_v. \vspace{0.1cm}
	\end{array}
	\label{eqOriginalMinProblem}
\end{equation}


\section{The feasible set of the minimization problem}\label{secTheFeasibleSet}

In this section, we will explore the topology and geometry of the feasible set of the optimization problem (\ref{eqOriginalMinProblem}). Let us denote this set by $\Omega(\pazocal{H})$. 

The first thing to know is that, in fact, this set can be empty. An example is shown in Fig. \ref{FigEmptyFeasibleSet}. Here, we have three leaves $A$, $B$, $C$ lying on one line and a representative of the hanging type is displayed in the picture.

If we denote the parent of $A$ and $B$ by $X$, then, in order to be in the feasible set, we have to have $|XA|=|XB|$, which means $X=(0,t)$. According to (\ref{eqHypExistence}), finding a feasible position for the root $R$ is possible only when $|XA|\leq|XC|$. But this can never happen, which means that there is no LED tree of the given hanging type.

\begin{figure}[h]
	\centering
	\includegraphics[width=0.4\textwidth]{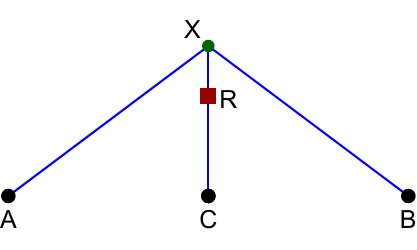}
	\caption{An example of a hanging type that does not contain any LED tree. The picture depicts a representative of the hanging type and we can see that $|XA|>|XC|$ for any feasible position of $X$. In the feasible set, we must have the opposite of that.}
	\label{FigEmptyFeasibleSet}
\end{figure}

If the set $\Omega(\pazocal{H})$ is not empty, then the minimum it can be is a half-line. This is not difficult to see: anytime we find a feasible position for the root, it comes with a whole hyperboloid sheet of other feasible positions. So even if we found just one feasible position for each non-root inner vertex, there would still be the hyperboloid sheet containing all possible roots. In the extreme situation when the distance of the foci is non-zero and equal to the difference of their heights, the corresponding hyperboloid sheet is a half line starting at the focus of the greater height. The direction of the half-line is given by the difference of this focus and the other focus (the one of the lower height). 

An example of such a situation in two dimensions is shown in Fig. \ref{FigFeasibleHalfLine}. In this case we have four leaves $A=(-a,0)$, $B=(a,0)$, $C=(0,a)$, $D=(0,-a)$. Again, we have $X=(0,t)$ and $|XA|\leq|XC|$, which implies $t\leq 0$. This, in turn, implies that $Y$ must be situated on the $x$-axis or below. But then we have $|YC|\geq |YD|$, while the feasibility condition requires $|YC|\leq |YD|$. This leaves us with only one possible placement of $X$ and $Y$, which is $X=Y=(0,0)$. The root $R$ can then be anywhere on the half-line $XC$ and hence the feasible set is a half-line in $\mathbb{R}^6$.

\begin{figure}[h]
	\centering
	\includegraphics[width=0.32\textwidth]{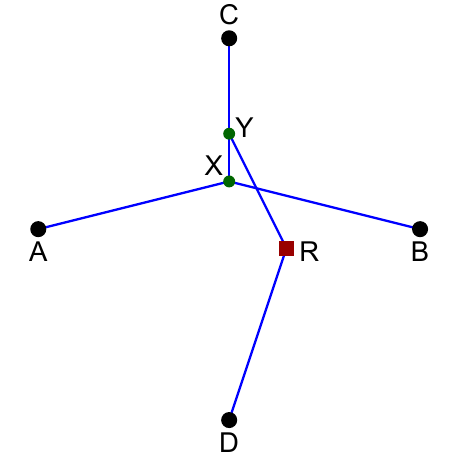}
	\includegraphics[width=0.32\textwidth]{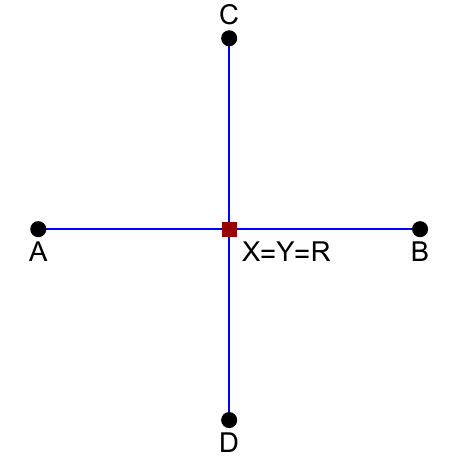}
	\includegraphics[width=0.32\textwidth]{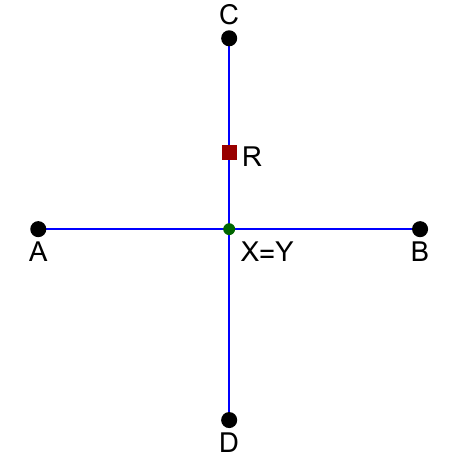} 
	\caption{An example of situation, when the feasible set is just one half-line. A representative of the hanging type is shown in the first picture. With this setting, there is only one possibility for placing $X$ and $Y$ -- the one we can see in the second and the third picture. The root $R$ can be anywhere on the half-line $XC$.}
	\label{FigFeasibleHalfLine}
\end{figure}

Before analyzing the geometry of the feasible set more in detail, let us devote some space to its connectedness. As we explained above, it cannot contain any isolated point -- each admissible root position brings with it a whole hyperboloid sheet of other possible roots. The character of the constraints and their finite number implies, that any component of connectedness will be at least path connected. Within this range, $\Omega(\pazocal{H})$ can have several types of topology -- it can be simply connected, path connected but not simply connected (i.e. it can have holes) or it can consist of two and more components of connectedness. Some examples are shown in Fig. \ref{FigConnectedness}. In all these examples, we have four leaves $A$, $B$, $C$, $D$, while $A$ and $B$ have a common parent $X$ and $C$ and $D$ have a common parent $Y$. The parent of $X$ and $Y$ is the root. Since the feasibility conditions imply $|XA|=|XB|$ and $|YC|=|YD|$, the point $X$ lies on the axis of the line segment $AB$ and the point $Y$ lies on the axis of the line segment $CD$. This means that the positions of $X$ and $Y$ can be parametrized by two parameters $s\in \mathbb{R}$, $t\in\mathbb{R}$.

The first example shows a setting where the feasible set is disconnected and consists of two components of connectedness. We have $A=(-a,0)$, $B=(a,0)$, $C=(-c,0)$, $D=(c,0)$, $a<0$, $c>0$, $c<a$. From the feasibility conditions we have $X=(0,s)$ and $Y=(0,t)$. To construct the feasible set, we have to find all pairs $(s,t)$ for which we are able to find an admissible root position.
In order for that to be possible, we must have
\[ |YX|\geq \big\vert |XA|-|YC|\big\vert, \]
which means
\begin{equation}
	|t-s|\geq \left\lvert\, \root\of{a^2+s^2}-\root\of{c^2+t^2}\right\rvert. 
	\label{eqDisconnectedness0}
\end{equation}
Let us explore the boundary of the set given by this inequality, i.e. the set where the corresponding equality holds. The terms on both sides are non-negative, therefore we get an equivalent inequality by squaring. It reads
\begin{equation}
	2\,\root\of{ (a^2+s^2)(c^2+t^2)}\geq a^2+c^2+2st. 
	\label{eqDisconnectedness1}
\end{equation}
Now we can square again, but later we will have to verify the non-negativity of the term on the right hand side. After squaring and some algebra, we get
\begin{equation}
	4c^2s^2-4(a^2+c^2)st+4a^2t^2-(a^2-c^2)^2=0. 
	\label{eqDisconnectedness2}
\end{equation}
This equality represents a quadratic curve and its discriminant is equal to $16\left(a^2-c^2\right)$. For $c<a$, this is a positive value and thus the curve is a hyperbola centered at the origin. The equality also implies
\[ 2st=\frac{4c^2s^2+4a^2t^2-\left(a^2-c^2\right)^2}{2\left(a^2+c^2\right)}. \]
Substituting this into the right hand side of (\ref{eqDisconnectedness1}), we get
\[ a^2+c^2+2st=\frac{a^4+6a^2c^2+c^4+4c^2s^2+4a^2t^2}{2\left(a^2+c^2\right)}>0. \]
Thus the hyperbola (\ref{eqDisconnectedness2}) satisfies the equality in (\ref{eqDisconnectedness1}) and it is the boundary of the set of admissible pairs $(s,t)$. Since for $(s,t)=(0,0)$, the inequality (\ref{eqDisconnectedness0}) is not satisfied, the admissible set is formed by the two parts of the plane, which are bounded by the hyperbola and contain its foci. Thus the set $\Omega(\pazocal{H})$ has to be disconnected and has two connected components.

The next example shows a situation when the feasible $(s,t)$-region is the whole plane. Here we have $A=(-a-d,0)$, $B=(-a+d,0)$, $C=(a-d,0)$, $D=(a+d,0)$, $a>0$, $d>0$ and $d<a$. Further, we have $X=(-a,s)$, $Y=(a,t)$. The condition that guarantees the existence of a corresponding LED tree reads
\[ \root\of{4a^2+(t-s)^2}\geq \left\lvert\, \root\of{d^2+s^2}-\root\of{d^2+t^2}\right\rvert. \]
By steps analogous to the previous case, we find that this condition is satisfied for all $(s,t)\in\mathbb{R}^2$. 

In the third example, the feasible $(s,t)$-region is simply connected, but it is not the whole plane. The coordinates of leaves are $A=(-1,0)$, $B=(1,0)$, $C=(-0.5,-0.5)$, $D=(0.5,-0.83)$.

Finally, we present an example when the feasible set has a hole. In this setting, we have $A=(-a,0)$, $B=(a,0)$, $C=(0,-c)$, $D=(0,c)$, $a>0$, $c>0$, $c\neq a$ and $X=(0,s)$, $Y=(t,0)$. The feasibility condition implies
\begin{equation}
	\root\of{t^2+s^2}\geq \left\lvert\, \root\of{a^2+s^2}-\root\of{c^2+t^2}\right\rvert, 
	\label{eqDisconnectedness4}
\end{equation}
which, by squaring, leads to solving the equation
\[ 2\,\root\of{(a^2+s^2)(c^2+t^2)}=a^2+c^2. \]
This, in turn, leads to
\[ s^2t^2+c^2s^2+a^2t^2-\frac{\left(a^2-c^2\right)^2}{4}=0. \]
The curve given by this equation is a quartic oval centered at the origin and it is displayed in the last picture of Fig. \ref{FigConnectedness}. Since the point $(0,0)$ does not satisfy the inequality (\ref{eqDisconnectedness4}), the feasible region for $(s,t)$ is the outside of the oval plus the oval itself. 

\begin{figure}[h]
	\centering
	\includegraphics[width=0.24\textwidth]{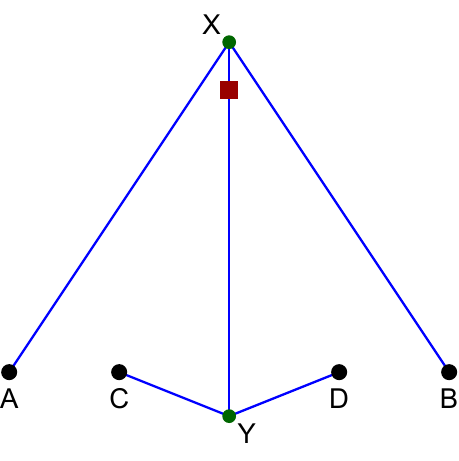}
	\includegraphics[width=0.24\textwidth]{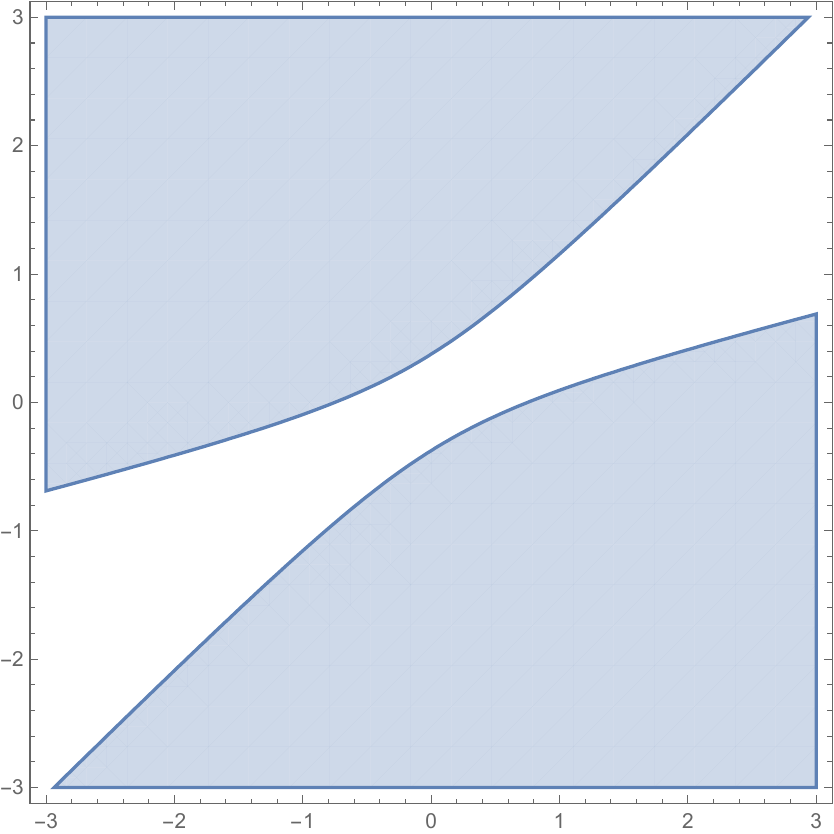}
	\includegraphics[width=0.24\textwidth]{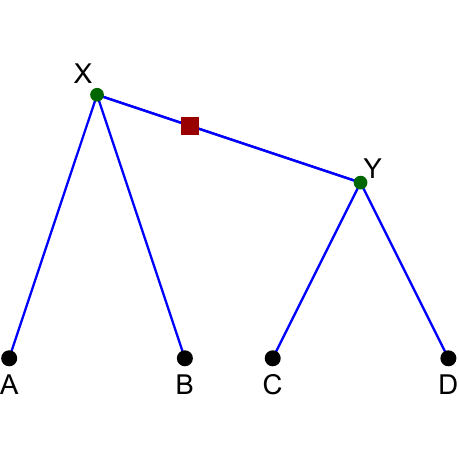} 
	\includegraphics[width=0.24\textwidth]{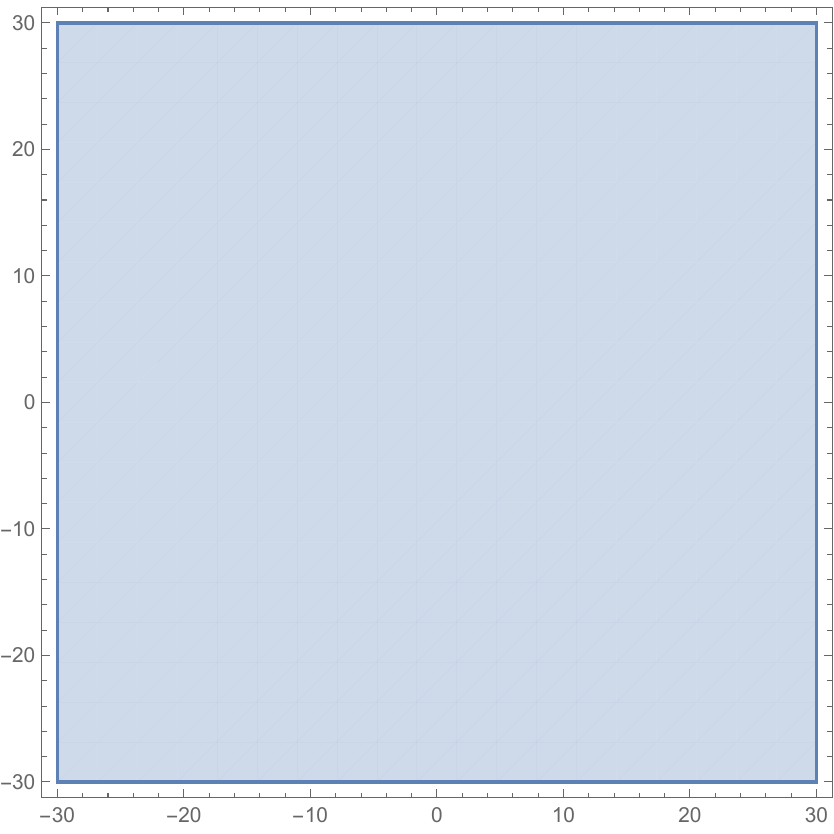} \vspace{0.1cm} \\
	\includegraphics[width=0.24\textwidth]{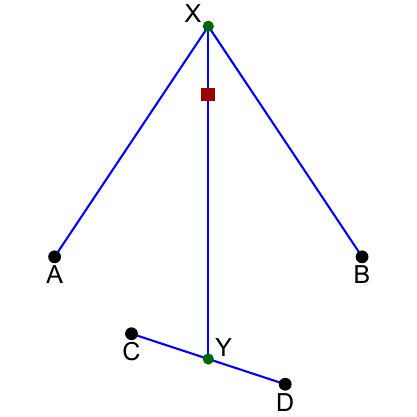}
	\includegraphics[width=0.24\textwidth]{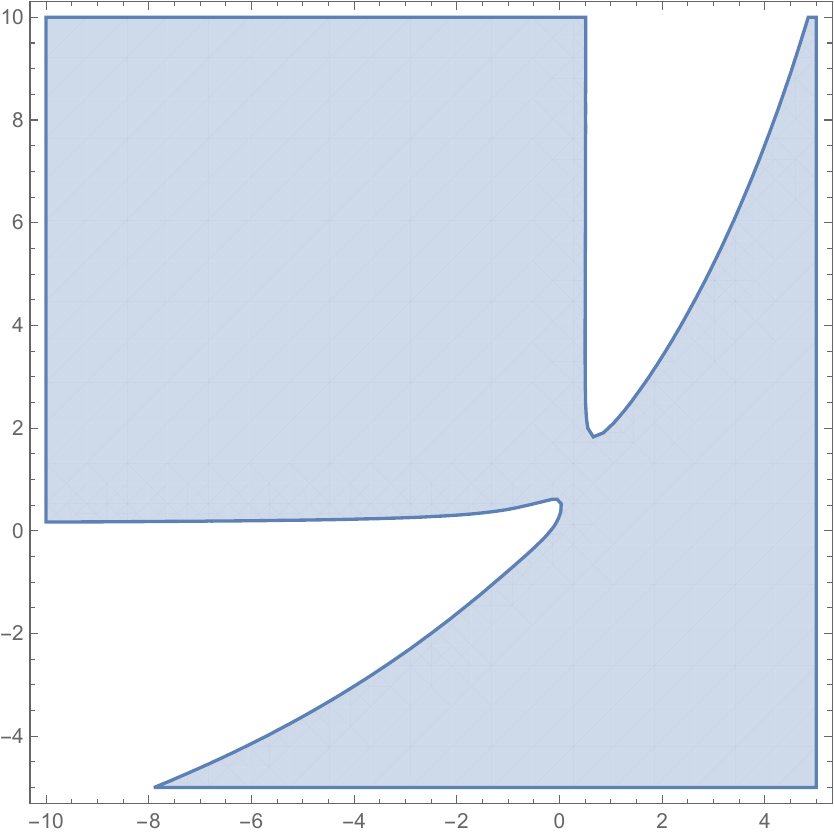}
	\includegraphics[width=0.24\textwidth]{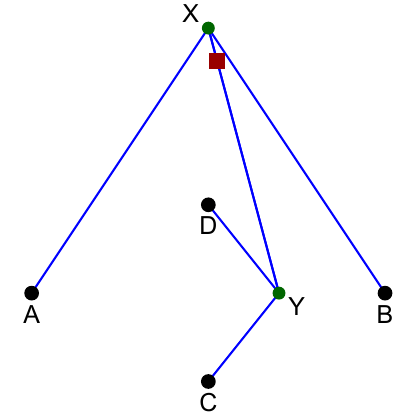}
	\includegraphics[width=0.24\textwidth]{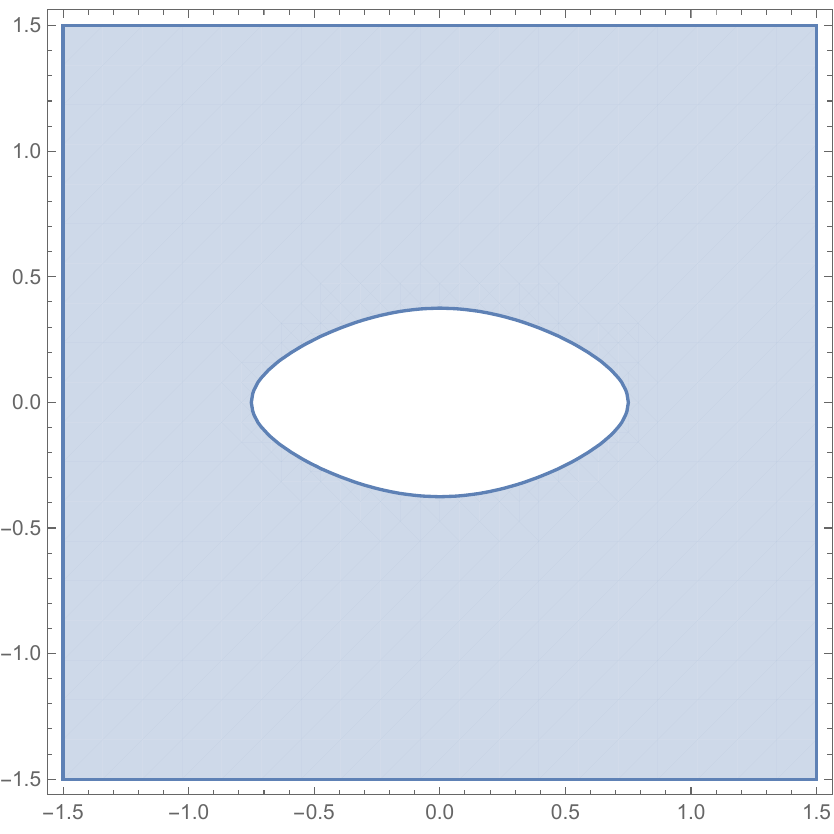}
	\caption{A demonstration of various types of connectedness of the feasible set. On the left of each picture, we have the leaves with one possible LED tree. On the right, we can see the feasible region for the parameter pair $(s,t)$ that parametrizes the lines, to which the inner vertices $X$ and $Y$ are constrained.}
	\label{FigConnectedness}
\end{figure}

The trees shown in Fig. \ref{FigConnectedness} are very simple, yet we get various types of connectedness of the feasible set. Finding sufficient conditions for connectedness is not a trivial task and currently it is a work in progress that will be the subject of a future publication. 

Now we can move forward and explore some other properties of the set $\Omega(\pazocal{H})$. To this aim, we introduce a regularity criterion for a point of this set.

\begin{dfn}\label{defRegularity}
	A point $\beta\in\Omega(\pazocal{H})$ is called {\em regular}, if $\Psi_{\beta}$ has the following properties: $\|e_j\|>0$, $j=1,\dots,n_e$, and $u_i^L\neq u_i^R$, $i=1,\dots,n_v$. A point that is not regular will be called a {\em singular} point. 
\end{dfn}
In simple words, we can say that in the LED tree corresponding to a regular point, no two adjacent vertices coincide (no edge has a length of zero) and the left and right leaf paths of any inner vertex leave the vertex in different directions. Some illustrations are shown in Fig. \ref{FigRegularSingular}.

\begin{figure}[h]
	\centering
	\includegraphics[width=0.3\textwidth]{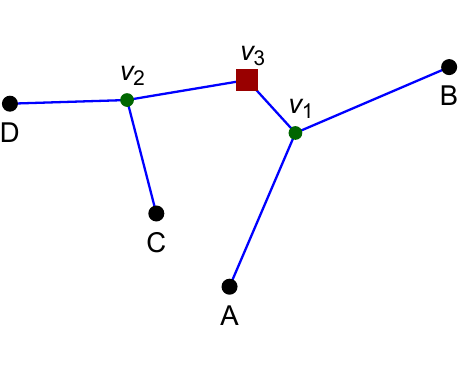} \hspace{0.1cm}
	\includegraphics[width=0.3\textwidth]{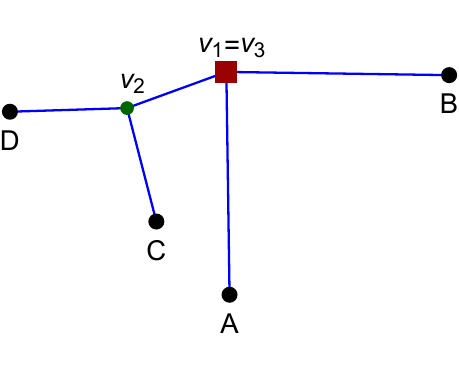} \hspace{0.1cm}
	\includegraphics[width=0.3\textwidth]{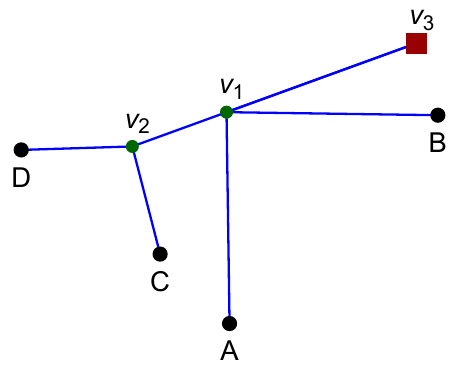} 
	\caption{Regular and singular points of the feasible set. In the first picture, we can see a LED tree corresponding to a regular point. In the second picture, the tree corresponds to a singular point. We can see that the inner vertex $v_1$ coincides with the root $v_3$. In the last picture, there are no coinciding vertices, but we have $u_3^L=u_3^R$.}
	\label{FigRegularSingular}
\end{figure}

It is useful to realize how the feasible set looks in a neighborhood of a regular point.  In the following part of the text, we are going to formulate and prove a proposition that specifies it.

\begin{conv}
	Let us have a set $S\subset\mathbb{R}^n$ and a point $s\in S$. Let $V$ be a neighborhood of $s$ in $\mathbb{R}^n$. Then the set $V\cap S$ will be called a {\em neighborhood} of $s$ {\em in $S$}.
\end{conv}

\begin{prop}\label{propParametrization}
	Let $\beta\in\Omega(\pazocal{H})$ be a regular point. Then there exists an open neighborhood of $\beta$ in $\Omega(\pazocal{H})$ that is a regular surface of dimension $n_v(n-1)$. 
\end{prop}
\begin{proof}
	In order to prove our claim, we will show that some neighborhood $\omega$ of $\beta$ in $\Omega(\pazocal{H})$ can be represented by a regular map $\sigma\colon U\rightarrow\mathbb{R}^{n_v n}$, where $U\in\mathbb{R}^{n_v(n-1)}$ is an open set. By a regular map we mean a map that is smooth, injective, has a continuous inversion and also linearly independent first derivatives.
	
	The proof will be made by induction on the number of leaves of $\Psi_{\beta}$. If $\Psi_{\beta}$ has only two leaves, the feasible set is a hyperplane and our statement is obviously true. Now let us suppose that it is true for any number of leaves up to $n_l-1$ and let $\beta$ be such that $\Psi_{\beta}$ has $n_l$ leaves. Let $v_1$ and $v_2$ be the children of the root of $\Psi_{\beta}$. As we have explained earlier, these vertices are the roots of two LED subtrees $\Psi_1$, $\Psi_2$ of $\Psi_{\beta}$. Let us first assume that both $\Psi_1$ and $\Psi_2$ have more than one leaf. Let $\Omega_1(\pazocal{H}_1)$ and $\Omega_2(\pazocal{H}_2)$ be the corresponding feasible sets and $\beta_1\in\Omega_1(\pazocal{H}_1)$ and $\beta_2\in\Omega_2(\pazocal{H}_2)$ the points representing $\Psi_1$ and $\Psi_2$. Since $\beta$ is regular, $\beta_1$ and $\beta_2$ must be regular as well. Let $n_1$, $n_2$ be the numbers of leaves of $\Psi_1$ and $\Psi_2$. The induction assumption says that there are neighborhoods $\omega_1$ of $\beta_1$ in $\Omega_1(\pazocal{H}_1)$ and $\omega_2$ of $\beta_2$ in $\Omega_2(\pazocal{H}_2)$ that can be parametrized by regular maps $\sigma_1\colon V_1\rightarrow\mathbb{R}^{n_1 n}$ and $\sigma_2\colon V_2\rightarrow\mathbb{R}^{n_2 n}$, where $V_1\in\mathbb{R}^{n_1(n-1)}$ and $V_2\in\mathbb{R}^{n_2(n-1)}$ are open sets. Let $t_0\in V_1$ and $s_0\in V_2$ be such that $\beta_1=\sigma_1(t_0)$ and $\beta_2=\sigma_2(s_0)$. Further, for any $t\in V_1$ and $s\in V_2$, let $h_1(t)$ and $h_2(s)$ be the heights of the trees corresponding to $\sigma_1(t)$ and $\sigma_2(s)$.
	
	Since $\beta$ is regular, the inequality (\ref{eqHypExistence}) must be sharp for $v_1$ and $v_2$. If it were not, we would either have $v_1=v_2$ or the root $v_r$ of $\Psi_{\beta}$ would lie on a hyperboloid collapsed to a half-line. Both cases lead to breaking the regularity conditions. 

Let us notice that both sides of the inequality (\ref{eqHypExistence}) are continuous with respect to the positions of $v_1$ and $v_2$. This means that there must be open neighborhoods $U_1\subset V_1$ of $t_0$ and $U_2\subset V_2$ of $s_0$ such that the inequality
\begin{equation}
	\|\sigma_1(t)-\sigma_2(s)\|>\left\lvert h_1(t)-h_2(s)\right\rvert 
	\label{eqSigma12Ineq}
\end{equation}
is satisfied for all $t\in U_1$ and $s\in U_2$. Thus for any $t\in U_1$ and $s\in U_2$, we can construct a non-degenerated  (not collapsed to a half-line) hyperboloid sheet $H(t,s)$ with the signed semi-major axis $a(t,s)$, semi-minor axis $b(t,s)$, center $c(t,s)$ and orientation $w(t,s)$ given as
\begin{subequations}\label{eqMapsabcew}
	\begin{align}
		a(t,s) & =  \frac{1}{2}\left(h_1(t)-h_2(s)\right) \nonumber \\
		\varepsilon(t,s) & =  \frac{1}{2}\|\sigma_1(t)-\sigma_2(s)\| \label{eqci} \\
		b(t,s) & =  \root\of{\varepsilon(s,t)^2-a(t,s)^2}, \label{eqbi} \\
		c(t,s) & =  \frac{1}{2}\left(\sigma_1(t)+\sigma_2(s)\right) \nonumber \\
		w(t,s) & =  \frac{\sigma_1(t)-\sigma_2(s)}{\|\sigma_1(t)-\sigma_2(s)\|} \label{eqwi}
	\end{align}
\end{subequations}
Now, let $U=U_1\times U_2\times\mathbb{R}^{n-1}$ and let us consider a map $\eta\colon U\rightarrow\mathbb{R}^{n}$ such that the partial application $\eta(t,s,\cdot)$ parametrizes the hyperboloid sheet $H(t,s)$. Since $H(t,s)$ is non-degenerated and $a$, $b$, $c$ and $w$ are continuous, we can assume that $\eta(t,s,\cdot)$ is regular. Using $\eta$, we define the map $\sigma\colon U\rightarrow\mathbb{R}^{n_v n}$ as
\[ \sigma(t,s,z)=(\sigma_1(t),\sigma_2(s),\eta(t,s,z)). \]
This map parametrizes a neighborhood of $\beta$ in $\Omega(\pazocal{H})$. Let us see, if it is regular. 

We already know that $\sigma_1$, $\sigma_2$ and $\eta(t,s,\cdot)$ are smooth maps. Thus the only thing that can spoil the smoothness of $\sigma$ are the differences in (\ref{eqci}), (\ref{eqbi}) and (\ref{eqwi}). But since $\beta$ is regular and the inequality (\ref{eqSigma12Ineq}) holds true in $U$, these differences are non-zero and all maps in (\ref{eqMapsabcew}) are smooth. Hence $\sigma$ is also smooth. The injectivity of $\sigma$ follows from the injectivity of $\sigma_1$ and $\sigma_2$. 

As for the continuity of the inverse, let us suppose that the inverse of $\sigma$ is not continuous. This means that there are some limit points $x_0$ and $y_0$ of $U$ such that $x_0\neq y_0$ and
\[ \lim\limits_{x\rightarrow x_0, y\rightarrow y_0}\|\sigma(x)-\sigma(y)\|=0. \]
But since $\sigma_1$, $\sigma_2$ and $\eta(t,s,\cdot)$ all have a continuous inverse, this is not possible. 

The last property to examine is the linear independence of the first derivatives of $\sigma$. The derivatives with respect to the components of $t$, $s$ and $z$ have the form
\begin{eqnarray}
	\frac{\partial\sigma}{\partial t^i} & = & \left(\frac{\partial\sigma_1}{\partial t^i}(t),{\bf 0},\frac{\partial\eta}{\partial t^i}(t,s,z)\right), \quad i=1,\dots,n_1, \label{eqSigmaDer1} \\
	\frac{\partial\sigma}{\partial s^j} & = & \left({\bf 0},\frac{\partial\sigma_2}{\partial s^j}(s),\frac{\partial\eta}{\partial s^j}(t,s,z)\right), \quad j=1,\dots,n_2, \label{eqSigmaDer2}\\
	\frac{\partial\sigma}{\partial z^k} & = & \left({\bf 0},{\bf 0},\frac{\partial\eta}{\partial z^k}(t,s,z)\right), \quad k=1,\dots, n-1,\label{eqSigmaDer3} 
\end{eqnarray}
where ${\bf 0}$ represents the vector of zeros of the corresponding dimension. The independence of the first derivatives of $\sigma_1$ guarantees the independence of the derivatives (\ref{eqSigmaDer1}). Similarly, the independence of the derivatives of $\sigma_2$ and $\eta(t,s,\cdot)$ implies the independence of the derivatives (\ref{eqSigmaDer2}) and (\ref{eqSigmaDer3}), respectively. From this, based on the position of zeros in the derivatives of $\sigma$, we can deduce that they are all linearly independent.  

Finally, when one of the children of the root of $\Psi_{\beta}$ is a leaf $v_i$, $i\in\{n_v+1,\dots,n_t\}$, the procedure and conclusions are analogous with the difference that $\sigma_2(s)$ is replaced by $v_i$ and $\sigma$ has only two parameters $t$ and $z$. 

\end{proof}

A singular point arises, if for some $v_i$ and its sibling $v_j$ we have
\[ \|v_i-v_j\|=|h(v_i)-h(v_j)|. \]
If $v_i=v_j$, then instead of a hyperboloid sheet, we will have a whole $n$-dimensional affine space. Thus in the corresponding singular points, $\Omega(\pazocal{H})$ is not locally homeomorphic to a Euclidean space. If $v_i\neq v_j$, the corresponding hyperboloid sheet collapses to a half line and its vertex becomes a cusp. As a consequence, a cusp will appear in $\Omega(\pazocal{H})$. Depending of the placement of the singular points, cusps might be also joined to ridges.

Two examples of how $\Omega(\pazocal{H})$ can look are shown in Fig. \ref{FigOmega}. Of course, in most cases it is impossible to visualize the feasible set because of its dimension and it is not easy to at least approximately imagine it. In Fig. \ref{FigOmega}, we show two very simple situations. In the first picture, we have a planar LED tree with only two inner vertices. In this case, $\Omega(\pazocal{H})$ is a two-dimensional surface in 4D. So, even though we cannot go any simpler, we are not able to fully display it. However, for the example that we present, one of the inner vertices ($X$) has its first coordinate equal to zero. That allows us to neglect this coordinate and show $\Omega(\pazocal{H})$ as a surface in 3D. As we can see, it has one cusp. It corresponds to the tree where $X=R$. The second example  shows a situation with four leaves and three inner vertices. To be able to see something, we fix one inner vertex ($Y$), so what we are showing is just a section of $\Omega(\pazocal{H})$. And again, the first coordinate of the vertex $X$ is zero. In this case, when $X=Y$, we will have singular points where $\Omega(\pazocal{H})$ is not locally homeomorphic to a Euclidean space.

\begin{figure}[h]
	\centering
	\includegraphics[width=0.35\textwidth]{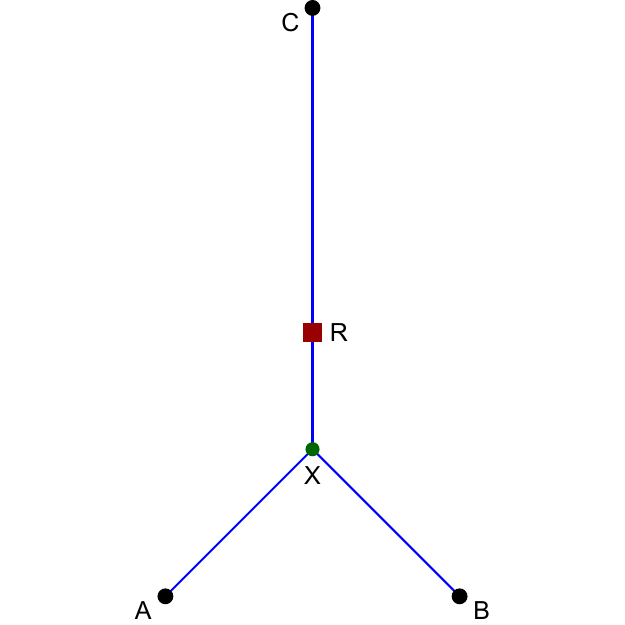} \hspace{0.3cm}
	\includegraphics[width=0.35\textwidth]{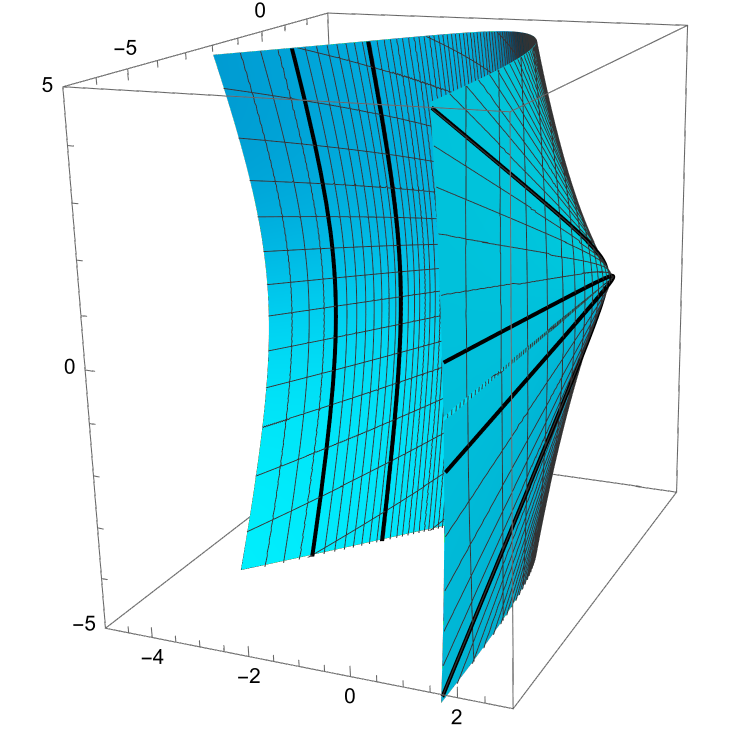}\vspace{0.1cm} \\
	\includegraphics[width=0.35\textwidth]{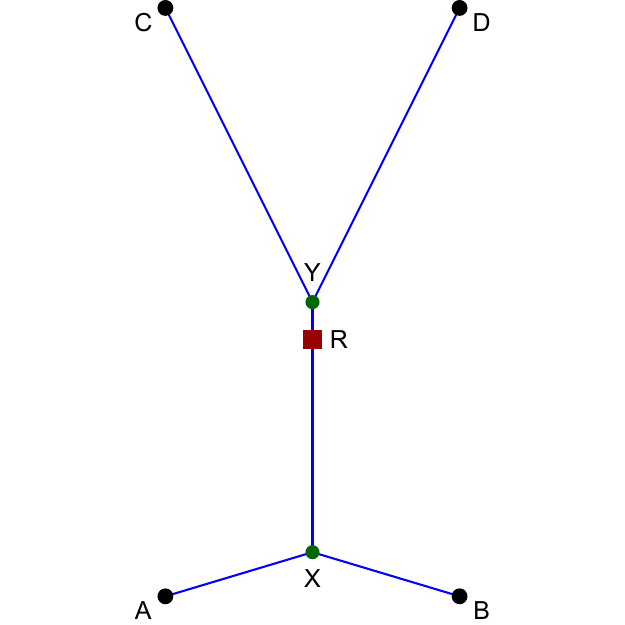}  \hspace{0.3cm}
	\includegraphics[width=0.35\textwidth]{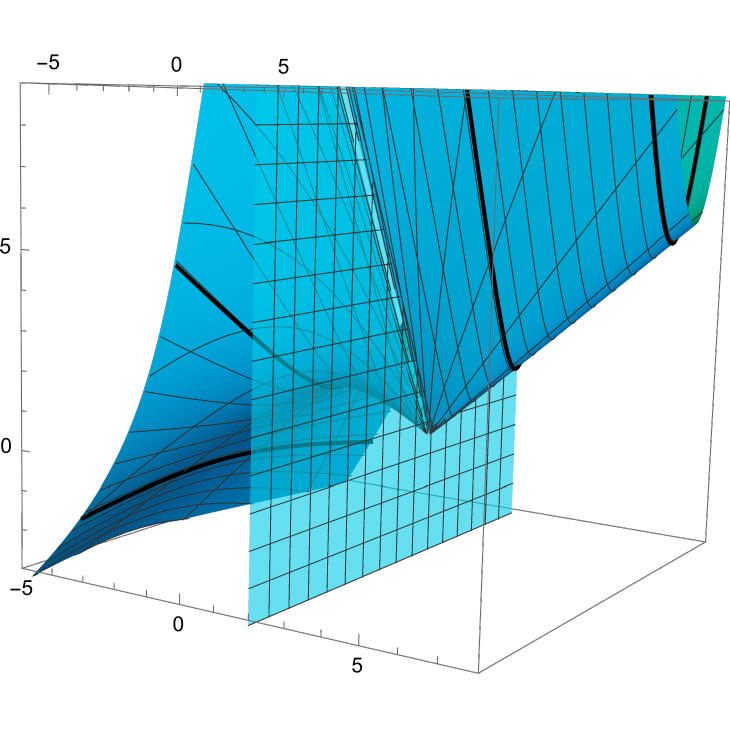} 
	\caption{Examples of the set $\Omega(\pazocal{H})$. In the first row, we show a situation with two inner vertices. The cusp corresponds to the tree where $X=R$. The rest of the singular points lie on the half line starting at the cusp. All other points are regular. In the second row, we present a case with three inner vertices. The vertex $Y$ is fixed, so we display only a section of $\Omega(\pazocal{H})$. The plane is a part of the set; it corresponds to the trees where $X=Y$. The cusp corresponds to the tree where $X=Y=R$. In both examples, we highlighted some curves -- those are some of the hyperbolas where the root $R$ can be situated.}
	\label{FigOmega}
\end{figure}

\subsection{The tangent space of $\Omega(\pazocal{H})$}\label{secMovingInsideOmega}

	Later on, in order to characterize stationary points of the objective function $\lambda$ in $\Omega(\pazocal{H})$, we will need its directional derivatives. We could obtain the available directions by differentiating the map $\sigma$ but, in fact, we have no explicit expression for $\sigma$. However, for a regular point $\beta$, the existence of $\sigma$ guarantees the existence of the tangent space $T_{\beta}\,\Omega(\pazocal{H})$ (as a linear space generated by the derivatives of $\sigma$) and its dimension. To be actually able to differentiate $\lambda$, we will now explore $T_{\beta}\,\Omega(\pazocal{H})$ from another point of view.
	 
	 Let us take any inner vertex $v_i$ of $\Psi_{\beta}$. As we already know, $v_i$ lies on a hyperboloid sheet $H_i$. Now let us imagine that we start moving $v_i$ within $H_i$. This, in general, cannot be done for free -- in order to stay in the feasible set, the other inner vertices will have to adjust to the changing position of $v_i$. However, not all vertices will have to move; the ones forced to react are only those on the root path of $v_i$. This follows directly from the construction of $\sigma$. 
	 
	 Let us explore how exactly will the affected vertices move. Let us consider the vertex $v_i$ and the following notation.
\begin{notn}\label{notRootPath}
	Let the root path of $v_i$ contain $m+2$ vertices $v_{i_0}=v_i$, $v_{i_1},\dots, v_{i_{m}}, v_{i_{m+1}}=v_r$. For the vertex $v_{i_k}$, $k=1,\dots, m+1$, let
\begin{itemize}
	\item{$e_{i_k}^+$ be the edge that is on the path and connects $v_{i_k}$ with its child,}
	\item{$e_{i_k}^-$ be the edge adjacent to $v_{i_k}$ that is not on the root path,}
	\item{$u_{i_k}^+$ and $u_{i_k}^-$ denote the unit vectors corresponding to $e_{i_k}^+$ and $e_{i_k}^-$, analogously to notation \ref{notBasic}.}
\end{itemize}
	 For $k=0$, we set $e_{i_0}^+=e_{i_L}$, $e_{i_0}^-=e_{i_R}$ and $u_{i_0}^+$ and $u_{i_0}^-$ accordingly. Note that $u_{i_{k-1}}^U=-u_{i_k}^+$.
\end{notn}
	 
	 Now let us say that the vertex $v_{i_{k-1}}$ moves with velocity $q_{i_{k-1}} \alpha_{i_{k-1}}$, where the direction $\alpha_{i_{k-1}}\in\mathbb{R}^n$ is a unit vector and $q_{i_{k-1}}\in\mathbb{R}$ is the speed of the movement. It is not difficult to determine the speed, at which the parent of $v_{i_{k-1}}$ -- the vertex $v_{i_k}$ -- will move. Let us suppose that we want $v_{i_k}$ to move in the direction $\alpha_{i_k}$, which is a unit vector, and the speed will be denoted by $q_{i_k}$. The feasibility condition says that all leaf paths of $v_{i_k}$ must maintain the equality of their lengths. In an arbitrary path containing $e_{i_k}^+$ and $e_{i_{k-1}}^-$, only the lengths of these two edges will be changed. In any path that contains $e_{i_k}^-$, the length of this edge will be the only one to be modified. The changes in these paths must be equal and thus we have
\[ q_{i_{k-1}}\alpha_{i_{k-1}}\cdot u_{i_{k-1}}^- +q_{i_{k-1}}\alpha_{i_{k-1}}\cdot u_{i_{k-1}}^U+q_{i_k}\alpha_{i_k}\cdot u_{i_k}^+=q_{i_k}\alpha_{i_k}\cdot u_{i_k}^-, \]
which implies
\begin{equation}
	q_{i_k}=\frac{q_{i_{k-1}}\alpha_{i_{k-1}}\cdot(u_{i_{k-1}}^-+u_{i_{k-1}}^U)}{\alpha_{i_k}\cdot(u_{i_k}^- -u_{i_k}^+)}. 
	\label{eqSpeedq}
\end{equation}
Of course, in order for the equality (\ref{eqSpeedq}) to make sense, we have to make sure that the denominator is not zero. Since we are at a regular point, we have $u_{i_k}^- - u_{i_k}^+\neq 0$. Thus the dot product of $\alpha_{i_k}$ and $u_{i_k}^- - u_{i_k}^+$ will be zero only if these two vectors are perpendicular to each other. To see when this happens, let us recall that $v_{i_k}$ lies on a hyperboloid sheet $H_{i_k}$. Let us first assume that $u_{i_k}^-\neq -u_{i_k}^+$ and let us consider the vector $u_{i_k}^- + u_{i_k}^+$. The vectors $u_{i_k}^-$ and $u_{i_k}^+$ together with $v_{i_k}$ determine a plane and the section of $H_{i_k}$ by this plane is a hyperbola that bisects the angle formed by $e_{i_k}^+$ and $e_{i_k}^-$. That means that $u_{i_k}^- + u_{i_k}^+$ is tangent to that hyperbola. The vector $u_{i_k}^- - u_{i_k}^+$ lies in the space spanned by $u_{i_k}^-$ and $u_{i_k}^+$ and is perpendicular to $u_{i_k}^- + u_{i_k}^+$. Therefore it is also normal to the mentioned hyperbola and to $H_{i_k}$ at $v_{i_k}$. In case when $u_{i_k}^-= -u_{i_k}^+$, the point $v_{i_k}$ is the vertex of $H_{i_k}$ and $u_{i_k}^- -u_{i_k}^+=2u_{i_k}^-$ is normal to $H_{i_k}$ at $v_{i_k}$ as well. The fact that $\alpha_{i_k}$ cannot be perpendicular to $u_{i_k}^- - u_{i_k}^+$ thus means that it cannot be tangent to $H_{i_k}$ at $v_{i_k}$. Otherwise it can be chosen arbitrarily.
	 	 
	Understanding what the movement of one vertex causes, we can construct a general tangent vector $\tau\in T_{\beta}\,\Omega(\pazocal{H})$. The procedure is illustrated in Fig. \ref{FigBasisOmega}. For an arbitrarily chosen $i\in\{1,\dots,n_v\}$, let us pick any vector $w_i$ from the tangent space $T_{v_i}H_i$ of $H_i$ at $v_i$ and let us set the components of $\tau$ in the $i$-th $(n-1)$-tuple to be equal to the components of $w_i$. Next, for $k=1,\dots,m$, the components in the $i_k$-th $(n-1)$-tuple will be assigned the values of the components of $q_{i_k}\alpha_{i_k}$, where $\alpha_{i_k}$ is an arbitrary admissible unit vector and $q_{i_k}$ is computed from (\ref{eqSpeedq}). The rest of the components will be set to zero. 
	
	Now let $\mathcal{T}(w_i)$ be the set of all vectors from $T_{\beta}\,\Omega(\pazocal{H})$ constructed from $w_i$ by the above procedure. It turns out that, under certain conditions, some of these vectors have a special property that will be useful at several further places and that could be called ``height preservation''. An example is shown in Fig. \ref{FigBasisOmega} (on the right). Let us consider the functions $h_i\colon\Omega(\pazocal{H})\rightarrow\mathbb{R}$, $i=1,\dots,n_v$, where $h_i(\beta)$ is the height of the $i$-th vertex of $\Psi_{\beta}$. Then the following lemma holds.
	
\begin{lem}\label{lemHP}
	Let $\beta$ be a regular point of $\Omega(\pazocal{H})$. Let $v_i$ be any inner vertex of $\Psi_{\beta}$ and let us use the notation \ref{notRootPath} for its root path. Let us have $p\in\{1,\dots,m+1\}$ and let us assume that $u_{i_k}^+\neq-u_{i_k}^-$ for all $k=1,\dots,p$. Then for any $w_i\in T_{v_i}H_i$, there is a tangent vector $\tau\in\mathcal{T}(w_i)$ such that
	\begin{equation}
		\nabla_{\tau}\,h_{i_k}(\beta)=0
		\label{eqDhij}
	\end{equation}
for all $j=1,\dots,p$.
\end{lem}

\begin{proof}
	Let us have an arbitrary $w_i\in T_{v_i}H_i$ and let $\alpha_{i_k}\in\mathbb{R}^{n-1}$, $k=1,\dots, p$, be unit vectors such that $\alpha_{i_k}\perp u_{i_k}^-$. Since $u_{i_k}^+\neq-u_{i_k}^-$, the vector $\alpha_{i_k}$ is not tangent to $H_{i_k}$ at $v_{i_k}$. Therefore all vectors $\alpha_{i_k}$, $k=1,\dots, p$, can be used for construction of a tangent vector $\tau\in\mathcal{T}(w_i)$ using (\ref{eqSpeedq}). Since $\alpha_{i_k}\perp u_{i_k}^-$, the corresponding directional derivative of the length of any leaf path of $v_{i_k}$ containing $e_{i_k}$ is zero. The feasibility conditions then imply (\ref{eqDhij}).
\end{proof}
	
\begin{figure}[h]
	\centering
	\includegraphics[width=0.4\textwidth]{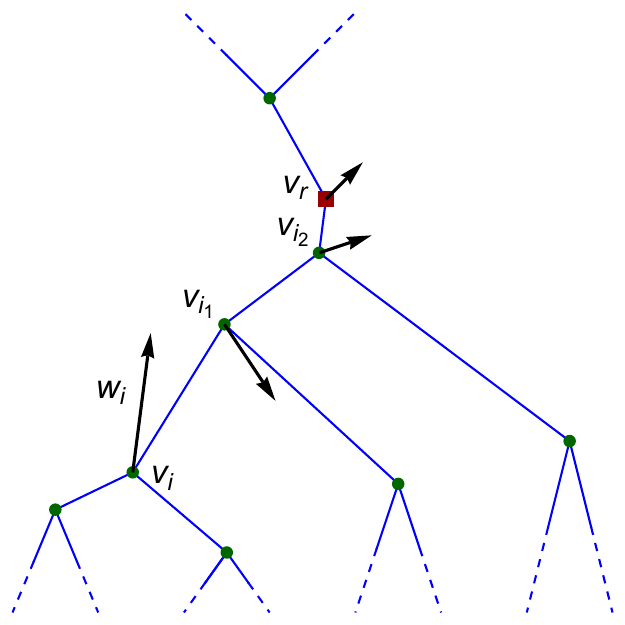} \hspace{0.1cm}
	\includegraphics[width=0.4\textwidth]{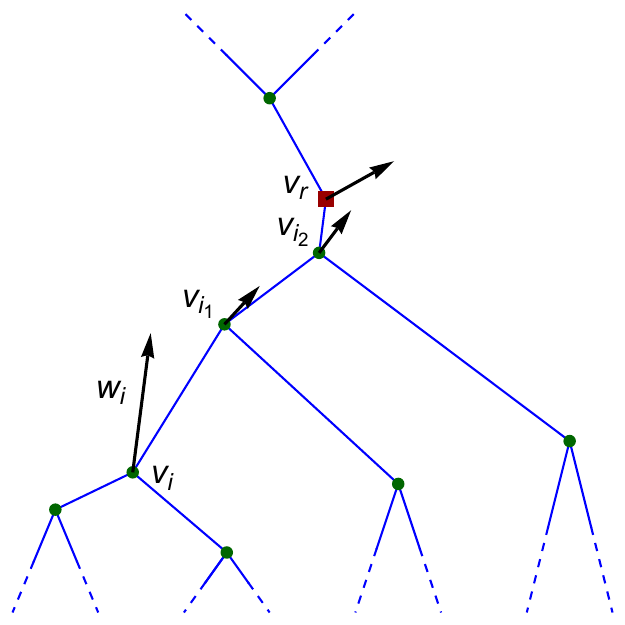} 
	\caption{Examples of vectors from a tangent space of the feasible set. On the left, we can see a vector where $w_i=u_i^L+u_i^R$. Moving $v_i$ with the velocity $w_i$ will provoke movement of all vertices on its root path. The directions of their velocities can be chosen arbitrarily as far as they are not tangent to the corresponding hyperboloid sheets. The speeds are computed according to (\ref{eqSpeedq}). The coordinates of these velocities are used as components of the corresponding tangent vector. On the right, we can see a height preserving direction according to Lemma \ref{lemHP}. The velocities are chosen so that the height of each inner vertex except $v_i$ remains unchanged.}
	\label{FigBasisOmega}
\end{figure}


\section{Stationary points of the objective function}\label{secStationaryPoints}

A regular point $\beta$ is a stationary point of $\lambda$ in $\Omega(\pazocal{H})$, if
\[ \nabla_{\tau}\lambda(\beta)=0 \]
for any $\tau\in T_{\beta}\Omega(\pazocal{H})$. 
It is not difficult to see that $\lambda$ cannot have a local maximum at $\beta$: whatever tree we have, we can always construct a tree of a higher length by moving the root appropriately within the hyperboloid sheet $H_r$. Other than that, we are not able to say much more about the types or number of the stationary points in this moment; we will get to this question later. 

In this section, we will discuss some important properties of LED trees corresponding to stationary points. As we will see, these trees have some specific characteristics. 

\begin{lem}\label{lemma1}
	Let $\beta$ be a stationary point of $\lambda$ on $\Omega(\pazocal{H})$. Then we have $u_r^L=-u_r^R$.
\end{lem}
\begin{proof}
	If $u_r^L\neq - u_r^R$, then the root of $\Psi_{\beta}$ is not situated in the vertex of the hyperboloid sheet $H_r$. If we move the root along the geodesic connecting it with the vertex, the value of $\lambda$ will decrease. Since there is a decreasing direction, $\beta$ cannot be a stationary point of $\lambda$.
\end{proof}

Lemma \ref{lemma1} says that the root of the tree corresponding to a stationary point always lies on the line segment connecting its children. Let this segment be denoted by $e_0$. 

For simplicity, several upcoming propositions will be formulated for a specific (but still very general) type of stationary point.

\begin{dfn}\label{def3dimStatPoint}
	A stationary point $\beta$ of $\lambda$ in $\Omega(\pazocal{H})$ is called {\em properly forked}, if $u_i^L$, $u_i^R$ and $u_i^U$ are three different vectors for each $i$, $i=1,\dots, n_v$, $i\neq r$.
\end{dfn}

Before we move on, let us introduce some other useful notations. 
\begin{notn}\label{notLambdas}
	Let us consider the following functions, all of them defined on $\Omega(\pazocal{H})$:
	\begin{itemize}
		\item{$\lambda_0\colon\Omega(\pazocal{H})\rightarrow\mathbb{R}$, where $\lambda_0(\beta)$ is the length of $e_0$ in $\Psi_{\beta}$,}
		\item{$\lambda_{i_k}^+$ and $\lambda_{i_k}^-$ defined for any $i=1,\dots,n_v$, where $\lambda_{i_k}^+(\beta)$ and $\lambda_{i_k}^-(\beta)$ are the lengths of the edges $e_{i_k}^+$ and $e_{i_k}^-$ on the root path of $v_i$ in $\Psi_{\beta}$.}
	\end{itemize}
\end{notn}

\begin{lem}\label{lemma2}
	Let $\beta$ be a properly forked stationary point of $\lambda$ on $\Omega(\pazocal{H})$. Then $u_i^L\neq -u_i^U$ and $u_i^R\neq -u_i^U$ for each $i=1,\dots,n_v$, $i\neq r$.
\end{lem}
\begin{proof}
	Let us suppose that $u_i^L=-u_i^U$ for some $i$. Let us consider a tangent vector $\tau$ constructed by the procedure presented in Sec. \ref{secMovingInsideOmega}, where $w_i=u_i^L+u_i^R$. Our assumption leads to
	\begin{equation}
		w_i\cdot u_i^L+w_i\cdot u_i^U=0, 
		\label{eqLemma2-1}
	\end{equation}
	which, based on (\ref{eqSpeedq}), implies that no other vertex will have to move in order for the tree to retain the LED property. Now, since $\beta$ is properly forked and we already have  $u_i^L=-u_i^U$, we cannot have $u_i^L=-u_i^R$. Therefore we cannot have $w_i\perp u_i^R$ and thus 
	\[ \nabla_{\tau}\lambda_{i_R}(\beta)\neq 0. \]
	But (\ref{eqLemma2-1}) means that 
	\[ \nabla_{\tau}\lambda_{i_L}(\beta)+\nabla_{\tau}\lambda_{i_U}(\beta)=0 \]
	and therefore
	\[ \nabla_{\tau}\lambda(\beta)=\nabla_{\tau}\lambda_{i_R}(\beta)\neq 0. \]
	This contradicts the assumption that $\beta$ is a stationary point. Analogously, we obtain that $u_i^R\neq -u_i^U$.
\end{proof}

For the sake of several proofs that follow, we will make use of Lemma \ref{lemHP} and we will define a specific type of height preserving direction. Let $\beta$ be a properly forked stationary point of $\lambda$ on $\Omega(\pazocal{H})$. Let us take an inner vertex $v_i$ of $\Psi_{\beta}$ and an arbitrary vector $w_i$ from $T_{v_i}H_i$. If $u_{i_k}^+\neq -u_{i_k}^-$ for $k=1,\dots,m$, then there are tangent vectors constructed from $w_i$ that preserve the height of all vertices on the root path of $v_i$ except $v_i$ itself and the root. Among these vectors, there is one that keeps the root on the line segment connecting its children. We will refer to this vector as $\tau_h(w_i)$. 

\begin{lem}\label{lemma3}
	Let $\beta$ be a properly forked stationary point of $\lambda$ on $\Omega(\pazocal{H})$. Then we have $u_i^L\neq -u_i^R$ for each $i=1,\dots,n_v$, $i\neq r$.
\end{lem}
\begin{proof}
	We prove our statement by induction on the depth of $v_i$. First, let $v_i$ be a child of the root and let us suppose that $u_i^L=-u_i^R$. Let us take a non-zero vector $w_i\in T_{v_i}H_i$ that is orthogonal to $u_i^L$ and $u_i^R$ but not to $u_i^U$; this is possible since $\beta$ is properly forked. From Lemma \ref{lemma1} follows that also $w_i\not\perp e_0$. Let $\tau=\tau_h(w_i)$. Then we have
	\begin{equation}
		\nabla_{\tau}\lambda_{i_L}(\beta)=\nabla_{\tau}\lambda_{i_R}(\beta)=0
		\label{eqLemma3-1}
	\end{equation}
	and 
	\[ \nabla_{\tau}\lambda_0(\beta)\neq 0. \]
	Since the length of $e_0$ will be the only one to change in the whole tree, we have 
	\[  \nabla_{\tau}\lambda(\beta)=\nabla_{\tau}\lambda_0(\beta)\neq 0. \]
	Thus $\beta$ cannot be a stationary point.
	
	Next, let us assume that $v_i$ is an arbitrary non-root vertex and that our statement is true for all vertices on its root path, starting from its parent. At the same time, let us have $u_i^L=-u_i^R$. Again, let $\tau=\tau_h(w_i)$ with $w_i$ chosen analogously to the initial case above. The existence of $\tau_h(w_i)$ is guaranteed by the induction assumption. Let us find out, what the value of $\nabla_{\tau}\lambda(\beta)$ is in this case.
	
	Let us invoke the notations \ref{notRootPath} and \ref{notLambdas} for the root path of $v_i$. From the properties of $\tau$ and from (\ref{eqLemma3-1}), we have
	\begin{equation}
		\nabla_{\tau}\lambda_{i_k}^-(\beta)=0
		\label{eqLemma3-2}
	\end{equation}
	for $k=1,\dots,m$. The feasibility conditions then imply that also
	 \begin{equation}
		\nabla_{\tau}\lambda_{i_k}^+(\beta)=0
		\label{eqLemma3-3}
	\end{equation}
	for $k=1,\dots,m$. However, since $\beta$ is properly forked and, according to Lemma \ref{lemma2}, we have $u_{i_k}^-\neq -u_{i_k}^U$ for all $k=0,\dots,m$, we get
	\begin{equation}
		\alpha_{i_k}\cdot u_{i_k}^U\neq 0,
		\label{eqAlphaijUij}
	\end{equation}
	for $k=1,\dots, m$. This means that all vertices $v_{i_k}$, $k=1,\dots,m$, will have to move. This includes the vertex $v_{i_{m}}$ and therefore we have
	\begin{equation}
	 	\nabla_{\tau}\lambda_0(\beta)\neq 0.
	 	\label{eqLambdacr}
	\end{equation}
	 Finally, as a consequence of (\ref{eqLemma3-1}), (\ref{eqLemma3-2}) and  (\ref{eqLemma3-3}), we get
	 \begin{equation}
	  	\nabla_{\tau}\lambda(\beta)=\nabla_{\tau}\lambda_0(\beta). 
	  	\label{eqNablaLambda}
	 \end{equation}
	  Since the right hand side is non-zero, $\beta$ cannot be a stationary point.
\end{proof}

\begin{lem}\label{lemma4}
	Let $\beta$ be a stationary point of $\lambda$ on $\Omega(\pazocal{H})$. Then the vectors $u_i^L$, $u_i^R$ and $u_i^U$ are a linearly dependent tripple for all $i=1,\dots, n_v$, $i\neq r$.
\end{lem}
\begin{proof}
	If $n=2$ or if $\beta$ is not properly forked, the statement is trivially true. So, let us consider that $\beta$ is properly forked and $n>2$. Let us take any non-root vertex $v_i$ and let us assume that $u_i^L$, $u_i^R$ and $u_i^U$ are linearly independent.
	
	Again, we set $\tau=\tau_h(w_i)$, but this time $w_i$ is such that $w_i\perp u_i^L+u_i^R$ but $w_i\not\perp u_i^U$. The existence of $\tau_h(w_i)$ follows from Lemma \ref{lemma3}. From here, we can use exactly the same arguments as in the induction step of the proof of Lemma \ref{lemma3}. Just as there, we will find out that $\nabla_{\tau}\lambda(\beta)\neq 0$ and therefore $\beta$ cannot be a stationary point.
\end{proof}

Let us summarize the results presented in this section. In a LED tree corresponding to a properly forked stationary point, the edges $e_{i_L}$, $e_{i_R}$ and $e_{i_U}$ lie in the same plane for any $i=1,\dots,n_v$. Moreover, if $i\neq r$, then no two of them are parallel. On the contrary, the edges $e_r^L$ and $e_r^R$ are parallel. An example of such a tree in 3D is shown in Fig. \ref{Fig3DLengthMinimizer}.

\begin{figure}[h]
	\centering
	\includegraphics[width=0.5\textwidth]{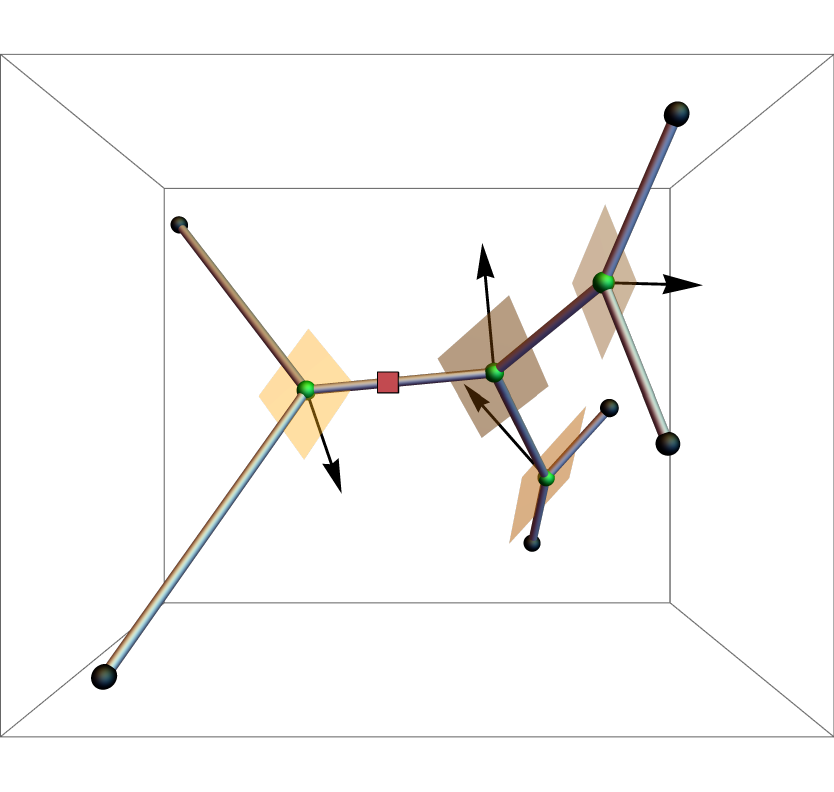} 
	\caption{An example of a three-dimensional tree corresponding to a properly forked stationary point. For each inner vertex, we show a piece of the plane where its adjacent edges lie and the normal to this plane. }
	\label{Fig3DLengthMinimizer}
\end{figure}


\section{A convex relaxation of the minimization problem}\label{secConvexRelaxation}

Coming back to the optimization problem (\ref{eqOriginalMinProblem}), we can see that the objective function $\lambda$ is a sum of convex functions and therefore it is also convex. However, as we can already imagine, the feasible set can look quite wild and is far from being convex. Because of this, the optimization problem seems rather complicated and it is not clear which method one should use. Also, one would naturally expect coming across local extrema and saddle points. Surprisingly, as we have already mentioned in the introductory part, stationary points that are not global minimizers are not really an issue. Local maxima were already excluded at the beginning of the previous section. To show the rest, we will use a convex relaxation of our problem.

The convex relaxation will use two variables: the vector $\beta\in\mathbb{R}^{n_v n}$ as in (\ref{eqOriginalMinProblem}) and in addition a vector $z\in\mathbb{R}^{n_e}$, $z=(z_1,\dots,z_{n_e})$. Let us use the same index sets $\pazocal{I}_i^L$ and $\pazocal{I}_i^R$ as in (\ref{eqOriginalMinProblem}).  Using this notation, we define convex functions $\varphi_i\colon\mathbb{R}^{n_e}\rightarrow\mathbb{R}$, $i=1,\dots, n_v$, and $f_j\colon \mathbb{R}^{n_v n}\times\mathbb{R}^{n_e}\rightarrow\mathbb{R}$, $j=1,\dots, n_e$, where
\begin{eqnarray}
	\varphi_i(z) & =& \sum\limits_{\iota\in\pazocal{I}_{i}^L} z_{\iota}-\sum\limits_{\iota\in\pazocal{I}_{i}^R} z_{\iota}, \label{eqPhi} \\
	f_j(\beta,z) & = & \lambda_j(\beta)-z_j. \label{eqf}
\end{eqnarray}
Then our convex relaxation reads
\begin{equation} 
	\begin{array}{ll}
		\displaystyle\min_{\beta\in\mathbb{R}^{n_v n},z\in\mathbb{R}^{n_e}} &\displaystyle\sum\limits_{j=1}^{n_e} z_j  \vspace{0.2cm}\\ 
		\mathrm{subject\,\, to} & \varphi_i(z)=0, \,\,i=1,\dots, n_v, \vspace{0.2cm} \\ 
		 & f_j(\beta,z)\leq 0, \,\, j=1,\dots,n_e.
	\end{array}
	\label{eqConvexRelaxation} 
\end{equation}
This problem can be interpreted as enlarging the feasible set by Euclidean representations of ``relaxed'' LED trees. These trees can be imagined as Euclidean LED trees with curved edges (Fig. \ref{FigRelaxedTree}). The variable $z_j$ represents the length of the $j$-th (curved) edge and it is always greater or equal to the distance of its endpoints.

\begin{figure}[h]
	\centering
	\includegraphics[width=0.35\textwidth]{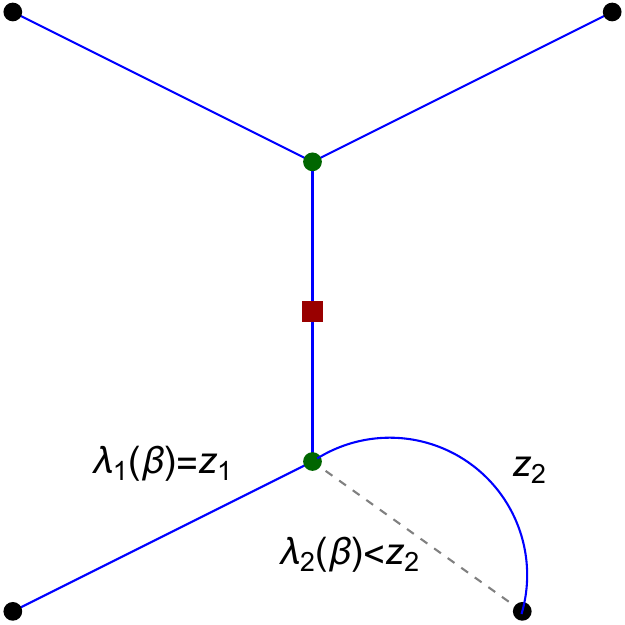} 
	\caption{A relaxed LED tree with one curved edge. The Euclidean representations of this type of trees define the feasible set of the convex relaxation (\ref{eqConvexRelaxation}).}
	\label{FigRelaxedTree}
\end{figure}

Denoting $\varphi(z)=(\varphi_1(z),\dots,\varphi_{n_v}(z))^T$ and $f(\beta,z)=(f_1(\beta,z),\dots,f_{n_e}(\beta,z))^T$, the Lagrange function $\mathcal{L}\colon\mathbb{R}^{n_v n}\times\mathbb{R}^{n_e}\times\mathbb{R}^{n_v}\times\mathbb{R}^{n_e}\rightarrow\mathbb{R}$ corresponding to this problem is defined as
\[ \mathcal{L}(\beta,z,x,y)=\sum\limits_{j=1}^{n_e} z_j+x^T\varphi(z)+y^Tf(\beta,z). \]
At a point $\beta$ where $\lambda_j(\beta)\neq 0$, $j=1,\dots,n_e$, the function $\mathcal{L}$ is differentiable with respect to all variables and we can formulate the corresponding Karush-Kuhn-Tucker (KKT) optimality conditions. Let $\beta_i$ denote the $i$-th $n$-tuple of $\beta$. Then the conditions have the form (for an explanation, see e.g. \cite[p. 267]{Boyd})
\begin{subequations}\label{eqKT}
	\begin{align}
		& \nabla_{\beta_i}\mathcal{L}=0, \,\, i=1,\dots n_v, \quad \frac{\partial\mathcal{L}}{\partial z_j}=0, \,\, j=1,\dots n_e, \label{eqKT1} \\
		& \varphi(z)=0, \label{eqKT4} \vspace{0.1cm}\\
		& f_j(\beta,z)\leq 0, \label{eqKT2} \vspace{0.1cm} \\
		& y^T f=0, \label{eqKT3} \vspace{0.1cm} \\
		& y_j\geq 0, \,\, j=1,\dots n_e. \label{eqKT5}
	\end{align}
\end{subequations}
Due to the convexity of the problem (\ref{eqConvexRelaxation}), the KKT conditions form a system of sufficient optimality conditions. Hence, if we find a point $(\tilde{\beta},\tilde{z},\tilde{x},\tilde{y})$ that satisfies the conditions (\ref{eqKT}), then $(\tilde{\beta},\tilde{z})$ is the optimal solution to the convex relaxation (\ref{eqConvexRelaxation}). If, moreover, all components of $\tilde{y}$ are positive, then $\tilde{\beta}$ is the optimal solution to the original problem (\ref{eqOriginalMinProblem}) .

Let us analyze the condition (\ref{eqKT1}) and let us see what exactly it implies. The first equality leads to
\begin{subequations}\label{eqKTy}
	\begin{align}
		\nabla_{\beta_r}\mathcal{L} & = y_{r_L}u_r^L+y_{r_R}u_r^R=0, \label{eqKT11r} \\
		\nabla_{\beta_i}\mathcal{L} & = y_{i_L}u_i^L+y_{i_R}u_i^R+y_{i_U}u_i^U=0,\,\, i=1,\dots, n_v, i\neq r. \label{eqKT11}
	\end{align}
\end{subequations}
As for the second equality, we have 
\[ \frac{\partial\mathcal{L}}{\partial z_j}=1+x^T\frac{\partial\varphi}{\partial z_j}-y_j=0 \]
for $j=1,\dots, n_e$. This can be rewritten as
\begin{equation}
	y_j=1+x^T\frac{\partial\varphi}{\partial z_j}.
	\label{eqKT12}
\end{equation}

The system of equations (\ref{eqKT12}) can be expressed in an equivalent form more suitable for out further reasoning, but we will need a little insight in the structure of the functions $\varphi_i$. From their definition (\ref{eqPhi}), we can see that $\frac{\partial \varphi_i}{\partial z_j}$ is either $0$, $1$, or $-1$. Now let us recall how we originally composed the constraints in the problem (\ref{eqOriginalMinProblem}) -- using the left and the right leaf paths. In both paths, we turn left immediately after the first edge. This means that any edge $e_{i_R}$, $i=1,\dots, n_v$, will appear only in the $i$-th constraint. Moreover, the length of $e_{i_R}$ is always in the role of a subtrahend. Hence the component $z_{i_R}$ will also appear only in $\varphi_i$ and we have $\frac{\partial \varphi_i}{\partial z_{i_R}}=-1$. The corresponding derivatives of the remaining components of $\varphi$ will be zero. This leads to the equality
\begin{equation}
	y_{i_R}=1-x_i.
	\label{eqYi_R}
\end{equation}
Further, we know that the edge $e_{r_L}$ also appears in only one constraint and $\|e_{r_L}\|$ is in the role of a summand. This gives us
\begin{equation}
	y_{r_L}=1+x_r.
	\label{eqYr_L}
\end{equation}
Finally, let us consider the edges $e_{i_L}$ and $e_{i_U}$ for any $i\neq r$. Again, because in each path we always turn left at any vertex, the length of the edge $e_{i_L}$ is included in all constraints that contain $e_{i_U}$ and always with the same sign. In addition, it also appears in the $i$-th constraint, always in the role of a summand. Therefore we have
\[ y_{i_L}=y_{i_U}+x_i. \]
Adding this with (\ref{eqYi_R}), we get
\begin{equation}
	y_{i_L}+y_{i_R}=y_{i_U}+1.
	\label{eqYi_L}
\end{equation}

Now if we summarize (\ref{eqYi_R})--(\ref{eqYi_L}), we obtain a system that is equivalent to (\ref{eqKT12}), namely
\begin{subequations}\label{eqY}
	\begin{align}
		& y_{r_L}+y_{r_R}=2, \label{eqY1} \\
		& y_{i_R}=1-x_i,\,\, i=1\dots, n_v, \label{eqY2} \\
		& y_{i_L}+y_{i_R}=y_{i_U}+1, \,\, i=1,\dots, n_v, \, i\neq r. \label{eqY3}
	\end{align}
\end{subequations}
As we can see, the system (\ref{eqY}) has $n_v+n_e$ unknowns $x_1,\dots, x_{n_v}$, $y_1,\dots,y_{n_e}$ and $2n_v=n_e$ equations just as the system (\ref{eqKT12}).

As a result of all the above reasoning, proving the KKT conditions (\ref{eqKT}) is equivalent to proving (\ref{eqKTy}), (\ref{eqY}) and (\ref{eqKT4})--(\ref{eqKT5}).


\section{Optimality of stationary points}\label{secOptimalityOfStationaryPoints}

\begin{prop}\label{propOptimalConvex}
	Let $\beta$ be a properly forked stationary point of $\lambda$ on $\Omega(\pazocal{H})$. Then it is an optimal solution to the problem (\ref{eqOriginalMinProblem}).
\end{prop}
\begin{proof}
	Let us set $z=(\lambda_1(\beta),\dots,\lambda_{n_e}(\beta))$. We will prove our statement by showing that there is $y\in\mathbb{R}^{n_e}$ and $x\in\mathbb{R}^{n_v}$ such that $(\beta,z,y,x)$ satisfies the KKT conditions (\ref{eqKT}) and that $y_j>0$ for any $j=1,\dots,n_e$.
	
	Setting $z$ as above, the conditions (\ref{eqKT4})--(\ref{eqKT3}) are trivially satisfied. Therefore it is enough to prove (\ref{eqKT1}), resp. (\ref{eqKTy}) and (\ref{eqY}), and (\ref{eqKT5}).
	
	Lemma \ref{lemma4} says that for $i\neq r$, the edges $e_{i_L}$, $e_{i_R}$ and $e_{i_U}$ lie in the same plane and no two of them are parallel. Therefore there are three non-zero values $y_{i_L}$, $y_{i_R}$ and $y_{i_U}$ such that 
	\begin{equation}
		 y_{i_L}u_i^L+y_{i_R}u_i^R+y_{i_U}u_i^U=0.
		 \label{eqLinCombY}
	\end{equation}
	Further, Lemma \ref{lemma1} claims that $u_r^L=-u_r^R$, which implies
	\begin{equation}
		y_{r_L}u_r^L+y_{r_R}u_r^R=0, 
		\label{eqLinCombRootY}
	\end{equation}
	where $y_{r_L}=y_{r_R}$ is an arbitrary non-zero value. Fixing this value and assuming that $v_i$ and $v_j$ are the children of the root, we can set $y_{i_U}=y_{r_L}$ and $y_{j_U}=y_{r_L}$ in (\ref{eqLinCombY}), which then uniquely determines the values $y_{i_L}$, $y_{i_R}$, $y_{j_L}$ and $y_{j_R}$. Proceeding analogously to the children of $v_i$ and $v_j$ and eventually to all vertices in the tree, we get a unique value $y_j$ corresponding to each edge $e_j$, $j=1,\dots, n_e$. 
	
	In our case, we need to satisfy the equality (\ref{eqY1}). Therefore we set $y_{r_L}=y_{r_R}=1$ and compute the remaining values $y_j$ based on this choice. The equations (\ref{eqLinCombY}) and (\ref{eqLinCombRootY}) are identical to the equations (\ref{eqKTy}) and thus the vector $y=(y_1,\dots, y_{n_e})$ might be what we need for satisfying the KKT conditions.

	The next thing that we will show is that $y=(y_1,\dots,y_{n_e})$ satisfies the equations (\ref{eqY3}). Let us consider any inner non-root vertex $v_i$ and the direction $\tau=\tau_h(w_i)$ (introduced in Sec. \ref{secStationaryPoints}) where $w_i=u_i^L+u_i^R$. Using the same reasoning as in the proof of Lemma \ref{lemma3} and the same notation as there and as in Section \ref{secMovingInsideOmega}, we find that (\ref{eqAlphaijUij}) and (\ref{eqLambdacr}) are true also in this case. Since this time $\nabla_{\tau} \lambda_{i_R}(\beta)\neq 0$, the equality (\ref{eqNablaLambda}) will change to
	\[ \nabla_{\tau}\lambda(\beta)=\nabla_{\tau}\lambda_{i_R}(\beta)+\nabla_{\tau}\lambda_0(\beta)=0. \]
	This means that
	\begin{equation}
		q_{i_{m}}\alpha_{i_{m}}\cdot u_{i_{m}}^U+w_i\cdot u_i^R=0.
		\label{eqNablaLambda0}
	\end{equation}
	Since $\tau$ is a height preserving direction, we have $\alpha_{i_k}\perp u_{i_k}^-$. Therefore (\ref{eqSpeedq}) reduces to
	\begin{equation}
		q_{i_k}=\frac{q_{i_{k-1}}\alpha_{i_{k-1}}\cdot u_{i_{k-1}}^U}{-\alpha_{i_k}\cdot u_{i_k}^+}=\frac{q_{i_{k-1}}\alpha_{i_{k-1}}\cdot u_{i_{k-1}}^U}{\alpha_{i_k}\cdot u_{i_{k-1}}^U}. 
		\label{eqSpeedqHP-r}
	\end{equation}
	Now let $y_{i_k}^+$ and $y_{i_k}^-$ denote the components of $y$ corresponding to the edges $e_{i_k}^+$ and $e_{i_k}^-$. Using (\ref{eqSpeedqHP-r}) and (\ref{eqKT11}), we get
	\begin{eqnarray*}
		q_{i_{m}}\alpha_{i_{m}}\cdot u_{i_{m}}^U & = & q_{i_{m}}\alpha_{i_{m}}\cdot (-y_{i_{m}}^+u_{i_{m}}^+-y_{i_{m}}^-u_{i_{m}}^-) 
		= - q_{i_{m}}\alpha_{i_{m}}\cdot(y_{i_{m}}^+ u_{i_{m}}^+) \\
		& = &  q_{i_{m}}\alpha_{i_{m}}\cdot (y_{i_{m-1}}^U u_{i_{m-1}}^U)  =   q_{i_{m-1}}\alpha_{i_{m-1}}\cdot (y_{i_{m-1}}^U u_{i_{m-1}}^U) \\
		& = &  q_{i_{m-1}}\alpha_{i_{m-1}}\cdot (-y_{i_{m-1}}^+u_{i_{m-1}}^+-y_{i_{m-1}}^-u_{i_{m-1}}^-) \\
		& = & \dots \\
		& = &  q_{i_{m-2}}\alpha_{i_{m-2}}\cdot (y_{i_{m-2}}^U u_{i_{m-2}}^U) \\
		& = & \dots \\
		& = &  q_{i_1}\alpha_{i_1}\cdot (y_{i_1}^U u_{i_1}^U).
	\end{eqnarray*}
	Finally, using $y_{i_1}^U u_{i_1}^U=-y_{i_1}^+u_{i_1}^+-y_{i_1}^-u_{i_1}^-$ and applying the non-reduced relation (\ref{eqSpeedq}), we get
	\[ q_{i_{m-1}}\alpha_{i_{m-1}}\cdot u_{i_{m-1}}^U = y_{i_U}w_i\cdot (u_i^R+u_i^U). \]
	Substituting this into (\ref{eqNablaLambda0}), we have
	\[ y_{i_U}w_i\cdot (u_i^R+u_i^U)+w_i\cdot u_i^R=0. \]
	This leads to
	\begin{equation}
		(y_{i_U}+1)w_i\cdot u_i^R = w_i\cdot (-y_{i_U}u_i^U)=w_i\cdot(y_{i_L}u_i^L+y_{i_R}u_i^R). 
		\label{eqYwu}
	\end{equation}
	Since $w_i\cdot u_i^L=w_i\cdot u_i^R$, we have
	\[  (y_{i_U}+1)w_i\cdot u_i^R = (y_{i_L}+y_{i_R})w_i\cdot u_i^R. \]
	According to lemma \ref{lemma3}, we have $u_i^L\neq-u_i^R$ and therefore $w_i\not\perp u_i^R$, which means $w_i\cdot u_i^R\neq 0$. This implies
	\[ y_{i_L}+y_{i_R}=y_{i_U}+1, \]
	which is exactly (\ref{eqY3}). 
	
	In order to fully satisfy the conditions (\ref{eqY}), we still need to find a suitable vector $x$. But this is trivial, since the components of $x$ appear only in (\ref{eqY2}) and can be obtained directly from there.
	
	The last thing to examine is the sign of the values $y_j$, $j=1,\dots,n_e$. So far we know just that they are non-zero, however, we need them to be positive. The proof of the positivity can be done by induction on the depth of the vertex $v_i$. Obviously, $y_{r_L}$ and $y_{r_R}$ (both equal to 1) are positive. So let us have any non-root vertex $v_i$ and let us assume that $y_{i_U}>0$. We want to show that $y_{i_L}$ and $y_{i_R}$ are then also positive.  
	
	Let $\varrho$ be the angle between $w_i$ and $u_i^R$ and $\vartheta$ the angle between $w_i$ and $-u_i^U$ (an illustration is shown in Figure \ref{FigPositiveY}). The first equality in (\ref{eqYwu}) can then be rewritten as
	\[ (1+y_{i_U})\left\|w_i\right\|\cos\varrho=y_{i_U}\left\|w_i\right\|\cos\vartheta, \]
	which implies
	\begin{equation}
		\cos\varrho=\frac{y_{i_U}}{1+y_{i_U}}\cos\vartheta. 
		\label{eqCosPhi}
	\end{equation}
	In general, the angle $\varrho$ is from the interval $\left[0,\frac{\pi}{2}\right]$. But for a properly forked stationary point, it cannot be zero and, according to Lemma \ref{lemma3}, it cannot be $\frac{\pi}{2}$ either. This means $\cos\varrho>0$. Since $y_{i_U}>0$, the relation (\ref{eqCosPhi}) implies that also $\cos\vartheta>0$. It also says that $\cos\varrho<\cos\vartheta$. This means that $\vartheta<\varrho$. This, in turn, implies that the coordinates of $-u_i^U$ in the basis $\{u_i^L,u_i^R\}$ have to be positive. As a consequence, both $y_{i_L}$ and $y_{i_R}$ must be positive.
	
	To conclude, we have found a quadruple $(\beta,z,x,y)$ that satisfies the KKT conditions (\ref{eqKT}). This means that $(\beta,z)$ is an optimal solution to the convex relaxation (\ref{eqConvexRelaxation}). And since the components of $y$ are positive, $\beta$ is an optimal solution to the original problem (\ref{eqOriginalMinProblem}).

\end{proof}

\begin{figure}[h]
	\centering
	\includegraphics[width=0.45\textwidth]{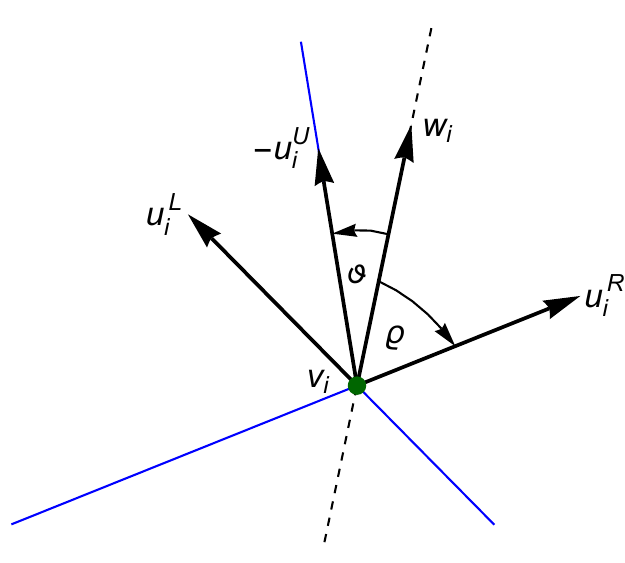} 
	\caption{An illustration to the proof that $y_j>0$, $j=1,\dots,n_e$. We find that $\vartheta<\varrho$ and thus $-u_i^U$ must have positive coordinates in the basis $\{u_i^L,u_i^R\}$.}
	\label{FigPositiveY}
\end{figure}

\begin{rem}
	Proposition \ref{propOptimalConvex} holds true also for stationary points that are not properly forked, i.e. when $u_k^L=u_k^U$ or $u_k^R=u_k^U$ for some inner vertex $v_k$. The proof can be done by a similar approach; the only problem is that the claims of Lemmas \ref{lemma2} and \ref{lemma3} need not be true and we can end up with $u_i^L$, $u_i^R$ and $u_i^U$ being all parallel. In that case, the height preserving direction $\tau_h(w_i)$ does not exist for any vertex $v_i$ whose root path contains $v_k$. Instead, we can differentiate $\lambda$ with respect to the direction constructed just as $\tau_h(w_i)$, with the only difference that $\alpha_k=u_k^-$ (and not $\alpha_k\perp u_k^-$). Otherwise the proof does not bring much new and is technically a bit tedious, so we do not show it here.
\end{rem}

\begin{prop}
	If $\lambda$ has a stationary point on $\Omega(\pazocal{H})$, then it is unique.
\end{prop}
\begin{proof}
	According to Proposition \ref{propOptimalConvex}, any stationary point of $\lambda$ on $\Omega(\pazocal{H})$ defines an optimal solution to a convex problem. Then the set of all stationary points must be convex. Let us assume that it contains two different points $A$, $B$. Then it also contains the line segment $AB$ and obviously the value of $\lambda$ is constant along $AB$. 
	
	Let $v_i(A)$ and $v_i(B)$, $i=1,\dots,n_t$, denote the $i$-th vertices of the trees $\Psi_A$ and $\Psi_B$ and let $v_i^U(A)$ and $v_i^U(B)$ be their parents. Let $k$ be such that $v_k(A)=v_k(B)$ but $v_k^U(A)\neq v_k^U(B)$; the existence of such $k$ is guaranteed since $\Psi_A$ and $\Psi_B$ have the same leaves, but they cannot have all vertices equal.   
	
Now recall that $\lambda$ is the sum of convex functions $\lambda_j$, $j=1,\dots,n_e$. From the choice of $k$ follows that the function $\lambda_{k_U}$, which is now just the distance function from a fixed point, must be strictly convex on $AB$. But this means that $\lambda$ must be strictly convex on $AB$ as well, which contradicts the assumption that the value of $\lambda$ is constant along $AB$. Hence the set of optimal solutions cannot contain two different points.

\end{proof}

\begin{rem}\label{remNoStationaryPoint}
	The existence of a stationary point of $\lambda$ on $\Omega(\pazocal{H})$ is not guaranteed. As we already know, the feasible set can be empty; but even when it is not, there need not be a stationary point. Some simple examples are shown in Fig. \ref{FigNoStationaryPoint}. On the first and the fourth picture, we can see length minimizers that correspond to stationray point. However, as we see in the rest of the pictures, if we move the leaf $A$ closer to the other leaves, at some point the length minimizer will stop being regular. The existence of a stationary point is an open problem so far.
\end{rem}

\begin{figure}[h]
	\centering
	\includegraphics[width=0.19\textwidth]{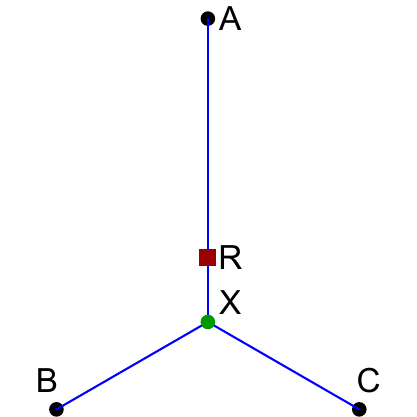}
	\includegraphics[width=0.19\textwidth]{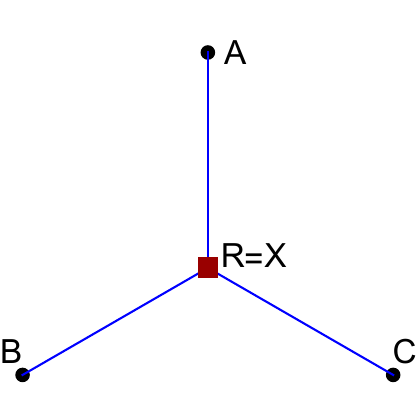}
	\includegraphics[width=0.19\textwidth]{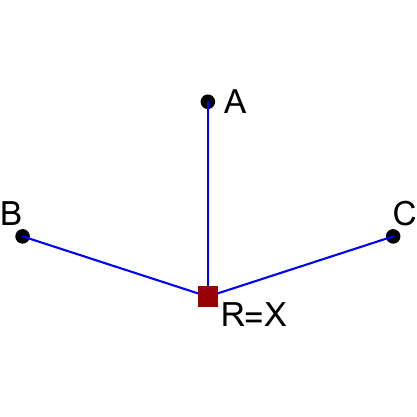} 
	\includegraphics[width=0.19\textwidth]{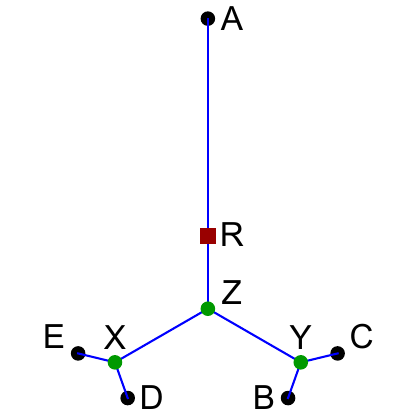} 
	\includegraphics[width=0.19\textwidth]{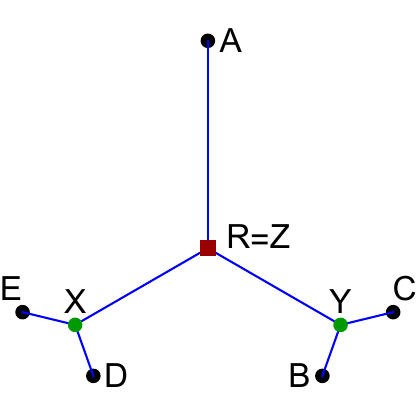} 
	\caption{On the first and the fourth picture, we can see length-minimizing LED trees that correspond to a stationary point. In the rest of the pictures, this is not the case and in fact, there is no stationary point of $\lambda$ on $\Omega(\pazocal{H})$.}
	\label{FigNoStationaryPoint}
\end{figure}


\section{Comparison of the length-minimizing LED tree with the Euclidean Steiner tree}\label{secComparisonLEDSteiner}

As we have already mentioned, if we have the same hanging type, the Euclidean Steiner tree and the length-minimizing LED tree can be equal, but in most cases they are not. However, there are some similarities and we add this short section to summarize them.

First of all, for a given hanging type, the Steiner tree problem is an unconstrained convex optimization problem and there is always one stationary point of the objective function. For the length-minimizing LED tree, we have a constrained non-convex problem and the stationary point does not have to exist, but if it does, it is unique as well.

Further, Lemmas \ref{lemma2}--\ref{lemma4} hold true not only for the length-minimizing LED tree but also for the Steiner tree. For the Steiner tree, there is one more well known geometrical property -- the angle between any pair from the tripple ($u_i^L$, $u_i^R$, $u_i^U$) is equal to $\frac{2\pi}{3}$. For a length-minimizing LED tree, this is not true anymore, but there are still some things that can be said. For example, let us suppose that $v_i$ is a child of the root. Then $y_{i_U}=1$ and from (\ref{eqCosPhi}) follows that
\[ \cos\varrho\leq \frac{1}{2}, \]
where $2\varrho$ is the angle between $u_i^L$ and $u_i^R$. This means that this angle is greater or equal $\frac{2\pi}{3}$. For the rest of the angles, we can obtain lower bounds that are less than $\frac{2\pi}{3}$, but this will be elaborated in a future work. 

From what was said above, it is also clear that in the Steiner tree, the vector $u_i^U$ is always parallel with $u_i^L+u_i^R$, namely
\[ -u_i^U=u_i^{LR}=\frac{u_i^L+u_i^R}{\|u_i^L+u_i^R\|}. \]
For the length-minimizing LED tree, this is not true, but $-u_i^U$ also cannot be bent away from $u_i^{LR}$ arbitrarily much. As we have seen in the proof of Proposition \ref{propOptimalConvex} (the equality (\ref{eqCosPhi}) and thereafter), in a properly forked length minimizer, the vector $-u_i^U$ lies in the relative interior of the cone defined by $u_i^L$ and $u_i^R$. In other words, if the angle between $u_i^L$ (or $u_i^R$) and $u_i^{LR}$ is $\varrho$, then the angle between $-u_i^U$ and $u_i^{LR}$ is less than $\varrho$. If the length-minimizing tree is not properly forked, then we have $\varrho=\vartheta=\frac{\pi}{2}$.


\section{Experiments}\label{secExperiments}

Since the concept of a length-minimizing LED tree emerged from a practical problem of evolution of language families, we complete the presentation of our work by showing some illustrative experiments. Of course, if one hopes to approach the reality, modeling the evolution of languages is a very complex task. It is also demanding when it comes to gathering the input data and deciding what they even should be. Therefore, the first thing to say is that the experiments presented here are not meant to answer the big questions yet, but are some of the first stage experiments that we performed in order to get acquainted with the length-minimizing LED trees, examine their potential in construction of chronograms and identify some further issues that should be tackled on the theoretical and practical level. The results that we present were obtained under some significant simplifications and we also limited ourselves to a rather small group of 18 Indo-European languages. Among these languages, there are 11 from the Slavic, 2 from the Baltic and 5 from the Romance family. 

\subsection{The main question}

One of the main questions of historical linguistics is to reconstruct the ancestors of the currently existing languages and to estimate when and where they were spoken. These problems caught the attention of scientist from different fields and mathematicians also came up with their own methods based on a number of different approaches (for example, Bouckaert et al. \cite{BA}, Chang et al. \cite{Chang}, Dyen et al. \cite{Dyen}, Gray and Atkinson \cite{GA}, Kassian et al. \cite{Kass}, Petroni and Serva \cite{PS}). The results obtained by different methods vary and therefore it is always interesting to bring in another one and see what it has to say. The approach based on LED trees could be potentially useful since it is quite easy to implement, not very time consuming and it models the actual evolution of languages in the feature space.   

\subsection{Placing the languages in a Euclidean space}\label{secFeatureSpace}

The first step in the procedure is choosing the feature space where the languages will be situated. In our case, for testing purposes, we chose a very simple approach, which is based on the Swadesh list. The Swadesh list \cite{Sw1,Sw2} is a list of 207 meanings that is often used by computational linguists for evaluation of differences between language. The meanings are chosen so that they belong to a very basic vocabulary of any language, including ancient languages (e.g. fire, sun, eye, water, soil etc.). The Swadesh list is usually used together with a cognate database -- the meanings are translated in each of the examined languages and if two translations of a given meanings are cognates (i.e. they have the same etymology), they are marked as equal. In our simplified case, we used just one translation per meaning and language. The cognate database was created with the help of the Swadesh lists on Wikipedia \cite{SwWiki}, various standard and etymological dictionaries (especially Slovak \cite{Kralik}, Czech \cite{Machek}, Italian \cite{ItaEty}, Lithuanian \cite{LitEty} and Russian \cite{RusEty}), and the database of Dunn \cite{Dunn,DunnLink}. We illustrate the procedure of setting the feature space coordinates on a simple example with only three languages and three meanings as shown in Table \ref{tabCognates}. 

The first step is to count all cognate groups across all given languages. In our example, the meaning ``night'' is said differently in each of the languages, but actually all three words are quite similar and have the same etymology. This makes one cognate group. On the other hand, for the meaning ``sky'', we have three completely unrelated words, which gives us three cognate groups. Finally, the meaning ``fire'' makes two cognate groups -- the Slovak and Lithuanian translations have the same etymology, while the Italian translation is not related to them. Altogether, we have six cognate groups and the feature space will be a six dimensional Euclidean space, where each coordinate represents one cognate group. In the basic setting, the $i$-th coordinate of a language is set to 1, if it has a word from the $i$-th cognate group. Otherwise, it is set  to 0. The resulting coordinates of our three languages are shown in the last column of Table \ref{tabCognates}.

This setting of coordinates is easy and straightforward, but it brings along a redundancy of coordinates. The resulting space has as many dimensions as there are cognate groups. However, if we have $n_l$ languages, the representing points will always define an $(n_l-1)$-dimensional simplex. Therefore, after the initial setting, we construct an isometric simplex in the $(n_l-1)$-dimensional Euclidean space and further computations are made with its vertices.  

\begin{table}
	\caption{An example of setting the feature space coordinates of three languages based on three meanings from the Swadesh list. }
	\label{tabCognates}
	\centering
	\begin{tabular}{l*{4}{l}}
		\toprule
		language              & night & sky & fire & coordinates \\
		\midrule
		Slovak 	& noc  & obloha & ohe\v n & (1 100 10)   \\
		Italian	& notte & cielo & fuoco & (1 010 01)   \\
		Lithuanian	& naktis  & dangus & ugnis & (1 001 10) \\
		\botrule
	\end{tabular}
\end{table}

\subsection{Determining the hanging type of the LED tree}\label{secTopology}

Having placed the languages in the feature space, we can proceed to determining a probable hanging type of the corresponding LED tree. This is done by a very simple iterative algorithm that combines smaller LED trees into larger ones. At the beginning, we take all languages in our batch and consider them to be LED trees with one leaf. Then, in each iteration,
\begin{enumerate}
	\item{we find the two closest roots of all current LED trees,}
	\item{we join these roots by a straight line and thus obtain a new (unrooted) tree,}
	\item{we insert a root in this tree so that we get a LED tree.}
\end{enumerate}
This procedure is repeated until we end up with a single LED tree. Since all inner vertices of this tree lie on straight segments connecting their children, we will calle this tree a {\it stretched} tree.

\subsection{Practical aspects of the computation and results}

Finally, we present several experiments and provide some observations and details as well as possible modifications of the basic procedure. In all presented experiments, we solved directly the original problem (\ref{eqOriginalMinProblem}). Relying on the fact that there is at most one stationary point, we simply used the FindMinimum function of Wolfram Mathematica that uses an interior point method. The initial approximation was the stretched tree used for the hanging type estimation. The CPU time was always in order of seconds.

One issue that we had to deal with is that the simple procedure proposed in Sec. \ref{secTopology} can fail for certain configurations of languages and it might be necessary to try several most probable topologies (i.e. not to always take the two closest tree roots, but sometimes the second or third closest). The conditions for existence of the stretched tree are yet to be examined; a comment on this will be made below.

In the first experiment, we took just 8 of the 18 languages. The stretched tree was constructed without any problems and yielded the hanging type that is illustrated in the picture. There was one stationary point and a 2D representation of the resulting tree is shown in Fig. \ref{FigSlavicBalticRomanSelect}. We also tried to use this tree to make some estimates of the splitting time of the individual language families. We made an assumption that the Romance languages split after the end of the Roman Empire, approximately 1550 years ago, and we computed the other splitting times based on this information and the heights of the corresponding inner vertices. The numbers that we obtained are shown in the picture next to the corresponding inner vertices. 

In the second experiment, we used all 18 languages. This time we had to make one adjustment in the hanging type -- based on the Swadesh list, the algorithm from Sec. \ref{secTopology} evaluated the common ancestor of Bulgarian, Macedonian and Serbian as a little closer to the central and east Slavic languagest than to Slovenian. For this setting, it was not possible to construct a corresponding stretched tree. Using the second most probable choice -- joining Slovenian with Bulgarian, Macedonian and Serbian -- yielded a regular stretched tree. However, the interior point method did not find any stationary point.

Analyzing the situation and recalling the examples from Fig. \ref{FigNoStationaryPoint} as well as other similar ones, we came to the conclusion that the Swadesh list alone does not create enough distance between different language families and subfamilies. Therefore we tried another setting of the feature space coordinates -- we assumed that the languages that differ in meanings that create just a few cognate groups should be situated further from each other than the ones that differ in meanings that are translated differently in almost every language. Specifically, instead of using just the values 0 and 1 as the coordinate values, we used 0 and $(c_{max}-c+1)^4$ where $c_{max}$ is the maximum number of cognate groups that one meaning generated for the 18 languages (it was the meaning ``dirty'' with 12 cognate groups) and $c$ is the number of cognate groups of the current word used for setting the coordinates. As we can see, the minimum non-zero coordinate will be 1 as before, but it can have also much larger values. Using this setting and the same hanging type as in the previous case, the stationary point was found and the resulting tree is depicted in Fig. \ref{FigSlavicBalticRoman}. Moreover, we also found out that the same adjustment reduces problems with finding the stretched tree. 

Finally, let us make a comment about the splitting times that we obtained. The estimates from other sources say that the Slavic languages split approximately 2200--1300 years before present, Latvian and Lithuanian 2200--1400 YBP, Slavic and Baltic languages 3800--2500 YBP and the common ancestor of Romance and Balto-Slavic languages existed approximately 4500--5700 YBP \cite{BA,Chang,GA,Kush,Nov,PS}. The length-minimizing LED tree should provide some kind of lower estimate of the splitting time, which agrees with the numbers we obtained for the splitting of Slavic, Baltic and Balto-Slavic language families, especially in the tree using all 18 languages. However, the common ancestor of the Romance and Balto-Slavic languages is dated later compared with the usual estimates. Even though our estimates are supposed to be on the lower side, we think that we would get a different number using a more elaborated feature space than the one based on Swadesh list with just one translation per meaning. For example, in the current setting, Latvian differs in 140 meanings from Macedonian and in 166 meanings from Portuguese. This is only 26 more different meanings out of 207 and it does not create enough space for a longer evolution from the common ancestor, even when using the adjustment presented in the previous paragraph. Combining this finding with the ones from the previous paragraphs, we can conclude that the method based on LED trees has a potential to produce reasonable results, however, some more complex criteria will have to be employed. For example, we could take into account synonyms, grammar features or actual resemblance of the words instead of/in addition to just etymology (for example Italian ``uovo'' and Serbian ``jaje'', both meaning ``egg'', come from the same Proto-Indo-European word, as well as the word ``egg'' itself, yet they all look completely different). Designing a suitable feature space will be one of the objects of our future practical work. 

\begin{figure}[h]
	\centering
	\includegraphics[width=0.45\textwidth]{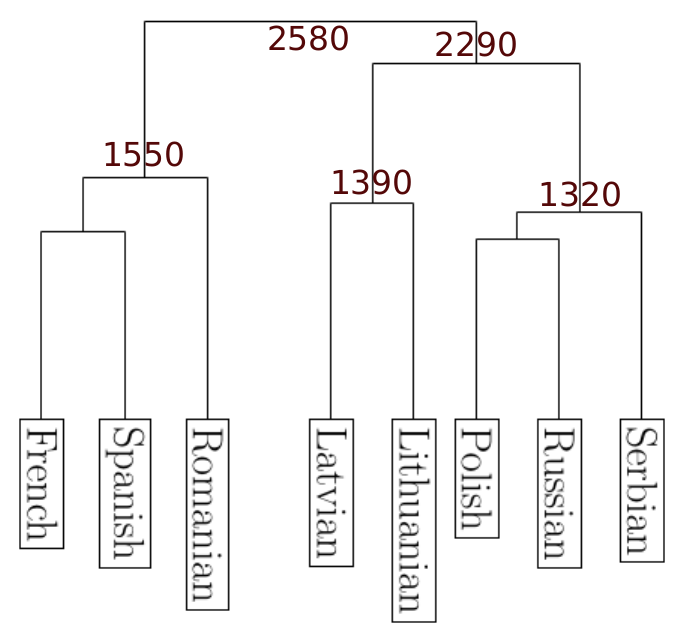} 
	\caption{A 2D chronogram corresponding to 8 selected Indo-European languages. The displayed splitting dates are in years before present.}
	\label{FigSlavicBalticRomanSelect}
\end{figure}

\begin{figure}[h]
	\centering
	\includegraphics[width=0.99\textwidth]{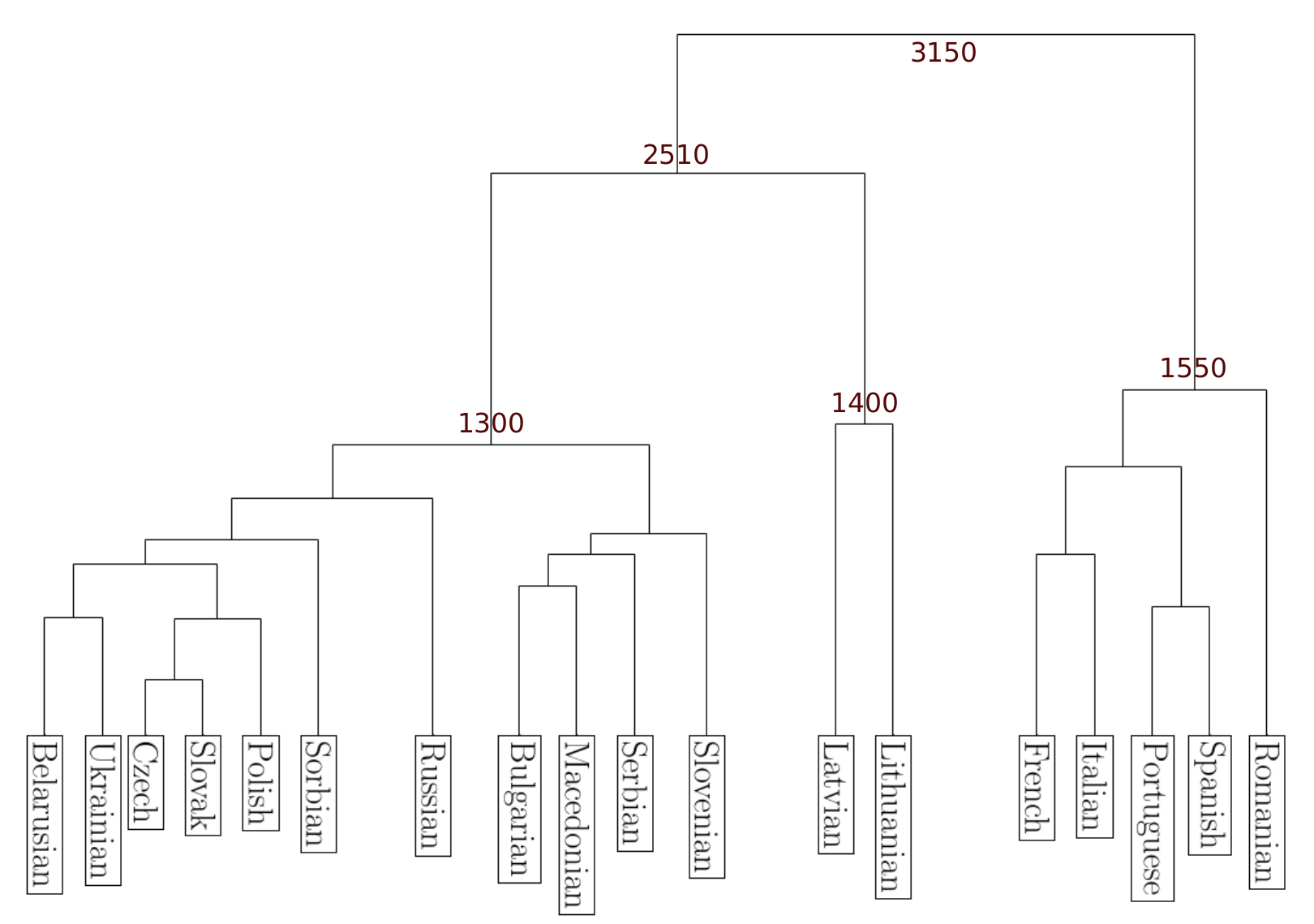} 
	\caption{A 2D chronogram corresponding to 18 selected Indo-European languages. The displayed splitting dates are in years before present.}
	\label{FigSlavicBalticRoman}
\end{figure}


\section*{Declarations}

The first author was supported by the grant APVV-19-0460, the second author by the grant VEGA 1/0036/23 and the third author by the grant APVV-20-0311. 
The authors have no relevant financial or non-financial interests to disclose.

\end{document}